\documentclass{amsart}
\usepackage{graphicx}
\newtheorem{proposition}{Proposition}

\newtheorem{theorem}[proposition]{Theorem}
\newtheorem{corollary}[proposition]{Corollary}
\newtheorem{lemma}[proposition]{Lemma}
\newtheorem{remark}[proposition]{Remark}

\def\ifm#1#2{\relax \ifmmode#1\else#2\fi}
\newcommand{\N}{\mathbb N}

\newcommand{\R}{\mathbb R}
\newcommand{\Rpos}{\mathbb{R}_{\!>0}}

\newcommand{\klk}    {\ifm {,\ldots,} {$,\ldots,$}}
\newcommand{\spar} {\vskip 0.25cm}

\begin{document}

\title[]{Efficient approximation of the solution of certain nonlinear reaction--diffusion equation II:\\ the case of large absorption}

\author[E. Dratman]{Ezequiel Dratman${}^{1}$}%

\address{${}^{1}$Instituto de Ciencias,
Universidad Nacional de Gene\-ral Sarmiento, Juan M. Gu\-ti\'e\-rrez
1150 (B1613GSX) Los Polvorines, Buenos Aires, Argentina.}
\email{edratman@ungs.edu.ar}
\urladdr{http://sites.google.com/site/ezequieldratman}

\thanks{Research was partially supported by the
following grants: UNGS 30/3084, UNGS 30/1066, CIC (2007--2009) and PIP 11220090100421.}


\maketitle
\date{\today}
%
%
\begin{abstract}
We study the positive stationary solutions of a standard finite-difference discretization of the semilinear heat equation with nonlinear Neumann boundary conditions. We prove that, if the absorption is large enough, compared with the flux in the boundary, there exists a unique solution of such a discretization, which approximates the unique positive stationary solution of the ``continuous'' equation. Furthermore, we exhibit an algorithm computing an $\varepsilon$-approximation of such a solution by means of a homotopy continuation method. The cost of our
algorithm is {\em linear} in the number of nodes involved in the discretization and the logarithm of the number of digits of approximation required.
\end{abstract}
\maketitle

%
\section{Introduction}\label{seccion: intro}
This article deals with the following semilinear heat equation with
Neumann boundary conditions:
\begin{equation}\label{ec: heat equation}
\left\{\begin{array}{rclcl}
  u_t & = & u_{xx}- g_1(u) & \quad & \mbox{in }(0,1)\times[0,T), \\
  u_x(1,t) & = & \alpha g_2\big(u(1,t)\big) & \quad & \mbox{in }[0,T), \\
  u_x(0,t) & = & 0 & \quad & \mbox{in }[0,T), \\
  u(x,0) & = & u_0(x)\ge 0 & \quad & \mbox{in }[0,1],
\end{array}\right.
\end{equation}
where $g_1,g_2 \in \mathcal{C}^3(\R)$ are analytic functions in $x=0$ and $\alpha$ is a positive constant. The nonlinear heat equation models many physical, biological and engineering phenomena, such as heat conduction (see, e.g., \cite[\S 20.3]{Cannon84}, \cite[\S 1.1]{Pao92}), chemical reactions and combustion (see, e.g., \cite[\S 5.5]{BeEb89}, \cite[\S 1.7]{Grindrod96}), growth and migration of populations (see, e.g., \cite[Chapter 13]{Murray02}, \cite[\S
1.1]{Pao92}), etc. In particular, ``power-law'' nonlinearities have long been of interest as a tractable prototype of general polynomial nonlinearities (see, e.g., \cite[\S 5.5]{BeEb89}, \cite[Chapter 7]{GiKe04}, \cite{Levine90}, \cite{SaGaKuMi95}, \cite[\S 1.1]{Pao92}).

The long-time behavior of the solutions of (\ref{ec: heat equation}) has been intensively studied (see, e.g., \cite{ChFiQu91}, \cite{LoMaWo93}, \cite{Quittner93}, \cite{Rossi98}, \cite{FeRo01}, \cite{RoTa01}, \cite{AnMaToRo02}, \cite{ChQu04} and the references therein). In order to describe the dynamic behavior of the solutions of (\ref{ec: heat equation}) it is usually necessary to analyze the behavior of the corresponding {\em stationary solutions} (see, e.g.,
\cite{FeRo01}, \cite{ChFiQu91}), i.e., the positive solutions of the following two-point boundary-value problem:
\begin{equation}\label{ec: heat equation - stationary}
\left\{\begin{array}{rclcl}
  u_{xx} & = & g_1(u) & \quad & \mbox{in }(0,1), \\
  u_x(1) & = & \alpha g_2\big(u(1)\big), & \quad & \\
  u_x(0) & = & 0. & \\
\end{array}\right.
\end{equation}

The usual numerical approach to the solution of (\ref{ec: heat equation}) consists of considering a second-order finite-difference discretization in the variable $x$, with a uniform mesh, keeping the variable $t$ continuous (see, e.g., \cite{BaBr98}). This semi-discretization in space leads to the following initial-value problem:
\begin{equation}\label{ec: heat equation - space discretiz}
\!\!\! \left
 \{\begin{array}{rcll}
 u_1'& = & \frac{2}{h^{2}}(u_2-u_1) - g_1(u_1), &  \\[1ex]
 u_k'& = & \frac{1}{h^2}(u_{k+1}-2u_k+u_{k-1})-g_1(u_k), \quad
 (2\le k\le n-1) \\[1ex]
 u_n' &\! = &\! \frac{2}{h^2}(u_{n-1}-u_n) - g_1(u_n) + \frac{2 \alpha}{h}g_2(u_n),
 \\[1ex]
 u_k(0) &\! = &\! u_0(x_k), \qquad  (1\le k\le n)
\end{array}\right.
\end{equation}
where $h:=1/(n-1)$ and $x_1,\dots,x_n$ define a uniform partition of the interval $[0, 1]$.
A similar analysis to that in \cite{DrMa09} shows the convergence of the positive solutions of (\ref{ec: heat equation
- space discretiz}) to those of (\ref{ec: heat equation}) and proves that every bounded solution of (\ref{ec: heat equation - space discretiz}) tends to a stationary solution of (\ref{ec: heat equation - space discretiz}), namely to a solution of
\begin{equation}\label{ec: heat equation - stationary space discretiz}
\left\{\begin{array}{rclcl}
  0& = & \frac{2}{h^2}(u_2-u_1) - g_1(u_1), & \quad & \\[1ex]
  0& = & \frac{1}{h^2}(u_{k+1}-2u_k+u_{k-1}) - g_1(u_k),
  \quad (2\le k\le n-1)\ \ \\[1ex]
  0 & = & \frac{2}{h^2}(u_{n-1}-u_n) - g_1(u_n) + \frac{2\alpha}{h} g_2(u_n).
 \\
\end{array}\right.
\end{equation}
Hence, the dynamic behavior of the positive solutions of (\ref{ec: heat equation - space discretiz}) is rather determined by the set of solutions $(u_1\klk u_n)\in(\Rpos)^n$ of (\ref{ec: heat equation - stationary space discretiz}).

Very little is known concerning the study of the stationary solutions of (\ref{ec: heat equation - space discretiz}) and the comparison between the stationary solutions of (\ref{ec: heat equation - space discretiz}) and (\ref{ec: heat equation}). In \cite{FeRo01}, \cite{DrMa09} and \cite{Dratman10} there is a complete study of the positive solutions of (\ref{ec: heat equation - stationary space discretiz}) for the particular case $g_1(x):=x^p$ and $g_2(x):=x^q$, i.e., a complete study of the positive solutions of
\begin{equation}\label{ec: heat equation - stationary space discretiz - monomial}
\left\{\begin{array}{rclcl}
  0& = & \frac{2}{h^2}(u_2-u_1) - u_1^p, & \quad & \\[1ex]
  0& = & \frac{1}{h^2}(u_{k+1}-2u_k+u_{k-1}) - u_k^p,
  \quad (2\le k\le n-1)\ \ \\[1ex]
  0 & = & \frac{2}{h^2}(u_{n-1}-u_n) - u_n^p + \frac{2\alpha}{h} u_n^q.
 \\
\end{array}\right.
\end{equation}
In \cite{FeRo01} it is shown that there are spurious solutions of (\ref{ec: heat equation - stationary space discretiz}) for $q < p <2q-1$, that is, positive solutions of (\ref{ec: heat equation - stationary space discretiz}) not converging to any solution of (\ref{ec: heat equation - stationary}) as the mesh size $h$ tends to zero.

In \cite{DrMa09} and \cite{Dratman10} there is a complete study of (\ref{ec: heat equation - stationary space discretiz}) for $p > 2q-1$ and $p<q$. In these articles it is shown that in such cases there exists exactly one positive real solution. Furthermore, a numeric algorithm solving a given instance of the problem under consideration with $n^{O(1)}$ operations is proposed. In particular, the algorithm of \cite{Dratman10} has linear cost in $n$, that is, this algorithm gives a numerical approximation of the desired solution with $O(n)$ operations.

We observe that the family of systems (\ref{ec: heat equation - stationary space discretiz - monomial}) has typically an exponential number $O(p^n)$ of {\em complex} solutions (\cite{DeDrMa05}), and hence it is ill conditioned from the point of view of its solution by the so-called robust universal algorithms (cf. \cite{Pardo00}, \cite{CaGiHeMaPa03}, \cite{DrMaWa09}). An example of such algorithms is that of general continuation methods (see, e.g., \cite{AlGe90}). This shows the need of algorithms specifically designed to compute positive solutions of ``structured'' systems like (\ref{ec: heat equation - stationary space discretiz}).

Continuation methods aimed at approximating the real solutions of nonlinear systems arising from a discretization of two-point boundary-value problems for second-order ordinary differential equations have been considered in the literature (see, e.g., \cite{AlBaSoWa06}, \cite{Duvallet90}, \cite{Watson80}). These works are usually concerned with Dirichlet problems involving an equation of the form $u_{xx}=f(x,u,u_x)$ for which the existence and
uniqueness of solutions is known. Further, they focus on the existence of a suitable homotopy path rather on the cost of the underlying continuation algorithm. As a consequence, they do not seem to be suitable for the solution of (\ref{ec: heat equation - stationary space discretiz}). On the other hand, it is worth mentioning the analysis of \cite{Kacewicz02} on the complexity of shooting methods for two-point boundary value problems.

Let $g_1,g_2 \in \mathcal{C}^3(\R)$ be analytic functions in $x=0$ such that $g_i(0)=0$, $g_i'(x) > 0$, $g_i''(x) > 0$ and $g_i'''(x) \ge 0$ for all $x>0$ with $i=1,2$. We observe that $g_1$ and $g_2$ are a wide generalization of the monomial functions of system (\ref{ec: heat equation - stationary space discretiz - monomial}). Moreover, we shall assume throughout the paper that the functions $g := g_1/g_2$ and $G := G_1/g_2^2$ are strictly increasing, where $G_1$ is the primitive function of $g_1$ such that $G_1(0)=0$, generalizing thus the relation $2q-1 < p$ in (\ref{ec: heat equation - stationary space discretiz - monomial}). In this article we study the existence and uniqueness of the positive solutions of (\ref{ec: heat equation - stationary space discretiz}), and we obtain numerical approximations of these solutions using homotopy methods. In \cite{Dratman11a} there is a complete study of (\ref{ec: heat equation - stationary space discretiz}) for $g:=g_1/g_2$ strictly decreasing. Furthermore, a similar analysis to that in \cite{FeRo01} shows that there are spurious solutions of (\ref{ec: heat equation - stationary space discretiz}) for $g$ strictly increasing and $G$ strictly decreasing; i.e., the generalization of the relations $q < p <2q-1$ in (\ref{ec: heat equation - stationary space discretiz - monomial}). According to these remarks, we have a complete outlook about the existence and uniqueness of the positive solutions of (\ref{ec: heat equation - stationary space discretiz}).

%
%
\subsection{Our contributions}
In the first part of the article we prove that (\ref{ec: heat equation - stationary space discretiz}) has a unique positive solution, and we obtain upper and lower bounds for this solution independents of $h$, generalizing the results of \cite{DrMa09}.

In the second part of the article we exhibit an algorithm which computes an $\varepsilon$-approximation of the positive solution of (\ref{ec: heat equation - stationary space discretiz}). Such an algorithm is a continuation method that tracks the positive real path determined by the smooth homotopy obtained by considering (\ref{ec: heat equation - stationary space discretiz}) as a family of systems parametrized by $\alpha$. Its cost is roughly of $n\log\log\varepsilon$ arithmetic operations, improving thus the exponential cost of general continuation methods.

The cost estimate of our algorithm is based on an analysis of the condition number of the corresponding homotopy path, which might be of independent interest. We prove that such a condition number can be bounded by a quantity independent of $h:=1/n$. This in particular implies that each member of the family of systems under consideration is significantly better conditioned than both an ``average'' dense system (see, e.g., \cite[Chapter 13, Theorem 1]{BlCuShSm98}) and an ``average'' sparse system (\cite[Theorem 1]{MaRo04}).

%
%
\subsection{Outline of the paper}

In Section \ref{subseccion: cotas} we obtain upper and lower bounds for the coordinates of the positive solution of (\ref{ec: heat equation - stationary space discretiz}).

Section \ref{subseccion: existencia_y_unicidad} is devoted to determine the number of positive solutions of (\ref{ec: heat equation - stationary space discretiz}). For this purpose, we prove that the homotopy of systems mentioned above is smooth (Theorem \ref{Teo: derivada alpha positiva}). From this result we deduce the existence and uniqueness of the positive solutions of (\ref{ec: heat equation - stationary space discretiz}).

In Section \ref{seccion: condicionamiento} we obtain estimates on the condition number of the homotopy path considered in the previous section (Theorem \ref{Teo: Cota Nro Condicion}). Such estimates are applied in Section \ref{seccion: algoritmo} in order to estimate the cost of the homotopy continuation method for computing the positive solution of
(\ref{ec: heat equation - stationary space discretiz}).

\spar
\spar

%

\section{Existence and uniqueness of stationary solutions}
\label{seccion: existencia y unicidad}
Let $U_1 \klk U_n$ be indeterminates over $\R$. Let $g_1$ and $g_2$ be two functions of class $\mathcal{C}^3(\R)$ such that $g_i(0) = 0$, $g_i'(x) > 0$, $g_i''(x) > 0$ y $g_i'''(x) \ge 0$ for all $x > 0$ with $i = 1,2$. As stated in the
introduction, we are interested in the positive solutions of (\ref{ec: heat equation - stationary space discretiz}) for a given positive value of $\alpha$, that is, in the positive solutions of the nonlinear system
\begin{equation}\label{ec: sistema_alpha}
\left\{\begin{array}{rclcl}
  0& = & -(U_2-U_1) + \frac{h^2}{2} g_1(U_1), \\[1ex]
  0& = & -(U_{k+1} - 2U_k + U_{k-1}) +  h^2 g_1(U_k),
  \quad (2\le k\le n-1)
  \\[1ex]
  0 & = & -(U_{n-1}-U_n)  \frac{h^2}{2} g_1(U_n) - h \alpha g_2(U_n),
  & \quad &  \\
\end{array}\right.
\end{equation}
for a given value $\alpha=\alpha^*>0$, where $h:={1}/(n-1)$. Observe that, as $\alpha$ runs through all
possible values in $\Rpos$, one may consider (\ref{ec: sistema_alpha})
as a family of nonlinear systems parametrized by $\alpha$, namely,
\begin{equation}\label{ec: sistema_A}
\left\{\begin{array}{rclcl}
  0& = & -(U_2-U_1) + \frac{h^2}{2} g_1(U_1), \\[1ex]
  0& = & -(U_{k+1} - 2U_k + U_{k-1}) +  h^2 g_1(U_k),
  \quad (2\le k\le n-1)
  \\[1ex]
  0 & = & -(U_{n-1}-U_n)  \frac{h^2}{2} g_1(U_n) - h A g_2(U_n),
  & \quad &  \\
\end{array}\right.
\end{equation}
where $A$ is a new indeterminate.

%
\subsection{Preliminary analysis}
Let $A,U_1\klk U_n$ be indeterminates over $\R$, set $U:=(U_1\klk
U_n)$ and denote by $F:\R^{n+1} \to \R^n$ the nonlinear map defined
by the right-hand side of (\ref{ec: sistema_A}). From the first
$n-1$ equations of (\ref{ec: sistema_A}) we easily see that, for a
given positive value $U_1=u_1$, the (positive) values of $U_2\klk
U_n,A$ are uniquely determined. Therefore, letting $U_1$ vary, we
may consider $U_2\klk U_n$, $A$ as functions of $U_1$, which are
indeed recursively defined as follows:
\begin{equation}\label{ec: relaciones recursivas}
\begin{array}{rcl}
 U_1(u_1)&:=& u_1,\\[1ex]
 U_2(u_1)&:=& u_1+\frac{h^2}{2}g_1(u_1),\\[1ex]
 U_{k+1}(u_1)&:=& 2U_k(u_1)-U_{k-1}(u_1)+h^2g_1\big(U_k(u_1)\big),\quad
(2\le k \le n-1),\\[1ex]
 A(u_1)&:=&
 \Big(\frac{1}{h}(U_n-U_{n-1})(u_1)+\frac{h}{2}g_1\big(U_n(u_1)\big)\Big)/ g_2\big(U_n(u_1)\big).
\end{array}
\end{equation}
Arguing recursively, one deduces the following lemma (cf.
\cite[Remark 20]{DeDrMa05}).
\begin{lemma}\label{Lema: positividad}
For any $u_1 > 0$, the following assertions hold:
\begin{enumerate}
\item \label{Lema: positividad - item 1}
$(U_k-U_{k-1})(u_1)=h^2\Big(\frac{1}{2} g_1(u_1)+\sum_{j=2}^{k-1}g_1\big(U_{j}(u_1)\big)\Big)
> 0$,
\item  \label{Lema: positividad - item 2} $U_k(u_1)=u_1 +
h^2\Big(\frac{k-1}{2}g_1(u_1) + \sum_{j=2}^{k-1}(k-j)g_1\big(U_{j}(u_1)\big)\Big)
> 0$,
\item  \label{Lema: positividad - item 3} $(U_k'-U_{k-1}')(u_1)
= h^2\big(\frac{1}{2}g_1'(u_1) + \sum_{j=2}^{k-1} g_1'\big(U_{j}(u_1)\big)
U_{j}'(u_1)\big) > 0$,
\item \label{Lema: positividad - item 4} $U_k'(u_1) = 1 +
h^2\big(\frac{k-1}{2}g_1'(u_1) + \sum_{j=2}^{k-1}(k-j)g_1'\big(U_{j}(u_1)\big)
U_{j}'(u_1)\big) > 1$,
\end{enumerate}
for $2\le k\le n$.
\end{lemma}

As in \cite{Dratman11a} we have the following lemma. This result is important for the existence and uniqueness of the positive solutions of (\ref{ec: sistema_A}).

\begin{lemma}\label{Lema: Crecimientos}
For any $u_1 > 0$, the following assertions hold:
\begin{enumerate}
\item \label{Lema: Crecimientos - item 1}
$\Big(\frac{U_k-U_{k-1}}{g_1(U_{k})}\Big)'(u_1)< 0$,
\item  \label{Lema: Crecimientos - item 2}
$\Big(\frac{U_k-U_{1}}{g_1(U_{k})}\Big)'(u_1)< 0$,
\item  \label{Lema: Crecimientos - item 3}
$\Big(\frac{U_{k}-U_{k-1}}{U_{k}-U_{1}}\Big)'(u_1) \ge 0$,
\item  \label{Lema: Crecimientos - item 4}
$\Big(\frac{g_1(U_k)}{g_1(U_1)}\Big)'(u_1) > 0$,
\end{enumerate}
for $2\le k\le n$.
\end{lemma}

The next result studies the monotony of certain relations between $g_1$ and $g_2$.

\begin{lemma}\label{Lema: Crecimientos 2}
Let $G_1$ be the primitive function of $g_1$ such that $G_1(0)=0$. If $x > 0$, then:
\begin{enumerate}
\item  \label{Lema: Crecimientos 2 - item 1}
$\Big(\frac{g_1^2}{G_1}\Big)'(x)> 0$.
\item  \label{Lema: Crecimientos 2 - item 2}
If $\Big(\frac{G_1}{g_2^2}\Big)'(x)>0$, then $\Big(\frac{g_1}{g_2}\Big)'(x)>0$.
\item  \label{Lema: Crecimientos 2 - item 3}
If exists $d \in [0,1)$ such that $\big(\ln (G_1^d(x)/g_2^2(x))\big)' \ge 0$, then $\Big(\frac{G_1}{g_2^2}\Big)'(x)>0$.
\end{enumerate}
\end{lemma}
\begin{proof} Since $g_1$ is a positive and strictly convex function in $\Rpos$ and $g_1(0)=0$, we have that
\begin{equation}\label{ec: desigualdad integral crecimientos}
\dfrac{g^2_1(x)}{2} = \int^x_0 g_1(t)g'_1(t) dt < g'_1(x) \int^x_0
g_1(t) dt = g'_1(x) G_1(x).
\end{equation}
Multiplying both sides by $g_1(x)$, we obtain
$$
\dfrac{g_1(x)}{G_1(x)} < \dfrac{2g_1(x)g'_1(x)}{g^2_1(x)},
$$
which proves (\ref{Lema: Crecimientos 2 - item 1}).

Now suppose that $(G_1/g_2^2)'(x)>0$, then
$$
\dfrac{2g_2(x)g'_2(x)}{g^2_2(x)} <  \dfrac{g_1(x)}{G_1(x)}.
$$
Combining this inequality with (\ref{ec: desigualdad integral crecimientos}), we obtain
$$
\dfrac{g'_2(x)}{g_2(x)} < \dfrac{g^2_1(x)}{2 G_1(x)g'_1(x)}
\dfrac{g'_1(x)}{g_1(x)} \le \dfrac{g'_1(x)}{g_1(x)},
$$
and (\ref{Lema: Crecimientos 2 - item 2}) is proved.

Finally, suppose that exists $d \in [0,1)$ such that $\big(\ln (G_1^d(x)/g_2^2(x))\big)' \ge 0$, we deduce
$$\begin{array}{rcl}
0 \le \big(\ln(G_1^d(x)/g_2^2(x))\big)' &=& d\big(\ln(G_1(x))\big)'-\big(\ln(g_2^2(x))\big)'\\
&=&\big(\ln(G_1(x))\big)'\Big(d-\dfrac{\big(\ln(g_2^2(x))\big)'}{\big(\ln(G_1(x))\big)'}\Big).
\end{array}$$
Since $\big(\ln(G_1(x))\big)'=g_1(x)/G_1(x) > 0$, we have that
$$
\dfrac{\big(\ln(g^2_2(x))\big)'}{\big(\ln(G_1(x))\big)'} \le d <
1.
$$
From this inequality, we obtain
\begin{equation}\label{ec: cond equiv G creciente}
\dfrac{2g'_2(x)g_2(x)}{g^2_2(x)}=\big(\ln(g^2_2(x))\big)' <
\big(\ln(G_1(x))\big)'=\dfrac{g_1(x)}{G_1(x)},
\end{equation}
which completes the proof.
\end{proof}

%
%

\subsubsection{Analogy between discrete and continuous solutions}
Set $u_k:= U_k(u_1)$ for $2\le k \le n$. The first step in our analysis of the positive solutions of (\ref{ec: sistema_A}) is to estimate the discrete derivative of the solution $u$ of (\ref{ec: heat equation - stationary}). Multiplying the identity $u''=g_1(u)$ by $u'$ and integrating over the interval $[0,x]$ it follows that
\begin{equation}\label{ec: continuous conservation law}
\frac{1}{2} u'(x)^2 =  \int_0^x u'(s)u''(s)ds = \int_0^x u'(s)g_1(u(s))ds = G_1(u(x))-G_1(u(0))
\end{equation}
holds for any $x \in (0,1)$, where $G_1$ is the primitive function of $g_1$ such that $G_1(0)=0$. The following result shows that $\frac{1}{2} (\frac{u_m-u_{m-1}}{h})^2$, a discretization of $\frac{1}{2}u'(x)^2$, equals the trapezoidal rule applied to $\int_0^x g_1(u(s))u'(s)ds$ up to a certain error term.
\begin{lemma}\label{Lema: derivada discreta vs trapecios}
For every $u_1>0$ and every $2\le m \le n$, we have:
\begin{equation}\label{ec: derivada discreta vs trapecios}
\begin{array}{rcl}
\dfrac{1}{2}\Big(\dfrac{u_m - u_{m-1}}{h}\Big)^2 &=&
 \displaystyle\sum_{k=1}^{m-1}\dfrac{g_1(u_{k+1}) + g_1(u_{k})}{2}(u_{k+1} - u_k)  \\[1ex]
  &&- \dfrac{g_1(u_1)}{4}(u_2 - u_1) - \dfrac{g_1(u_m)}{2}(u_m - u_{m-1}).
\end{array}
\end{equation}
\end{lemma}
\begin{proof}
Fix $u_1>0$ and $2\le m \le n$. For $m=2$ the statement holds by (\ref{ec: relaciones recursivas}). Next suppose that $m > 2$ holds. From (\ref{ec: relaciones recursivas}), we deduce the following identities
$$
\begin{array}{rcl}
 \dfrac{2(u_2- u_1)}{h^2}&=& g_1(u_1),\\[1ex]
 \dfrac{u_{k+1} - 2u_k + u_{k-1}}{h^2} &=&  g_1(u_k),\quad (2\le k \le m-1).
 \end{array}
$$
Multiplying the first identity by $(u_{2}-u_{1})/4h$ and the $k$th identity by $(u_{k+1}-u_{k-1})/2h$, we obtain
$$\begin{array}{rcl}
\dfrac{1}{2h} \Big(\dfrac{u_2 \!-\! u_1}{h}\Big)^2 \!\!\!\!&=&\!\!\!\! \dfrac{1}{4}\Big(\dfrac{u_{2} \!-\! u_{1}}{h}\Big) g_1(u_1),\\[2ex]
\dfrac{1}{2h}\Big(\Big(\dfrac{u_{k+1} \!-\! u_k}{h}\Big)^2 - \Big(\dfrac{u_k \!-\! u_{k-1}}{h}\Big)^2\Big) \!\!\!\!&=&\!\!\!\!
\dfrac{1}{2} \Big(\dfrac{u_{k+1} \!-\! u_k}{h} + \dfrac{u_k \!-\! u_{k-1}}{h}\Big) g_1(u_k),
\end{array}$$
for $2\le k\le m-1$. Note that $(u_{k+1}-u_{k-1})/2h$ is a numerical approximation of $u'((k-1)h)$ for $2\le k\le m-1$. Adding these identities multiplied by $h$ we obtain:
$$
\begin{array}{rrr}
\dfrac{1}{2}\Big(\dfrac{u_m \!-\! u_{m-1}}{h}\Big)^2 \!\!\!\!&=&\!\!\!\!
\dfrac{g_1(u_1)}{4}(u_2 \!-\! u_1)+
\displaystyle\sum_{k=2}^{m-1}\dfrac{g_1(u_k)}{2}(u_{k+1} \!-\! u_k+u_k-u_{k-1}) \\[2ex]
\!\!\!\!&=&\!\!\!\! \displaystyle\sum_{k=1}^{m-1}\frac{g_1(u_k)+g_1(u_{k+1})}{2}(u_{k+1} \!-\! u_k) - \dfrac{g_1(u_1)}{4}(u_2 \!-\! u_1)\\
&&  - \dfrac{g_1(u_m)}{2}(u_m \!-\! u_{m-1}).
\end{array}$$
This finishes the proof of the lemma.
\end{proof}

Substituting $x$ for 1 in (\ref{ec: continuous conservation law}) we
obtain the following identity (cf. \cite[\S 3]{ChFiQu91}):
\begin{equation}\label{ec: continuous energy conservation}
\frac{1}{2}\alpha^2 g^2_2(u(1))=G_1(u(1))-G_1(u(0)).
\end{equation}
From this identity one easily deduces that $u(1)$ is determined in terms of $\nu_0:=u(0)$, say, $u(1)=f(\nu_0)$.
Therefore, by (\ref{ec: continuous energy conservation}) it is possible to restate (\ref{ec: heat equation - stationary}) as an
initial-value problem, namely
$$\left\{\begin{array}{rclcl}
  u_{xx} & = & g_1(u) &   \mbox{in }(0,1), \\
  u(0) & = & \nu_0, &   \\
  u_x(0) & = & 0, & \\
\end{array}\right.$$
where $\nu_0>0$ is the solution of the equation $u_x(1)=\alpha g_2(f(\nu_0))$. Our purpose is to obtain a
discrete analogue of this identity, which will be crucial to determine the values $u_1$ of the positive solutions of (\ref{ec: sistema_A}).

Let $(\alpha,u):=(\alpha,u_1 \klk u_n) \in (\R_{>0})^{n+1}$ be a solution of (\ref{ec: sistema_A}). From the last equation of (\ref{ec: sistema_A}) we obtain
$$\begin{array}{rcl}
\dfrac{1}{2}\Big(\dfrac{u_n \!-\! u_{n-1}}{h}\Big)^2 \!\!\!\!&=&\!\!\!\! \dfrac{1}{2} \Big(\alpha g_2(u_n) \!-\! \dfrac{h}{2}g_1(u_n)\Big)^2 \\[1.5ex]
   \!\!\!\!&=&\!\!\!\! \dfrac{1}{2} \alpha^2 g^2_2(u_n)+\dfrac{h}{2} g_1(u_n)\Big( \dfrac{h}{2}g_1(u_n) \!-\! \alpha g_2(u_n) \Big) \!-\! \dfrac{h^2}{8}g^2_1(u_n)\\[1.5ex]
   \!\!\!\!&=&\!\!\!\! \dfrac{1}{2}\alpha^2 g^2_2(u_n)
       \!-\! \dfrac{h}{2} g_1(u_n)\Big(\dfrac{u_n \!-\! u_{n-1}}{h}\Big) \!-\! \dfrac{h^2}{8}g^2_1(u_n).
\end{array}$$
Combining this identity with Lemma \ref{Lema: derivada discreta vs trapecios} we obtain
$$
\dfrac{1}{2}\alpha^2 g^2_2(u_n) =
\displaystyle\sum_{k=1}^{n-1}\frac{g_1(u_k)+g_1(u_{k+1})}{2}(u_{k+1}-u_k) - \dfrac{g_1(u_1)}{4}(u_2-u_1) + \dfrac{h^2}{8}g^2_1(u_n).
$$
Using the identity $g_1(u_1)(u_2-u_1)=\frac{h^2}{2}g^2_1(u_1)$, we deduce that
$$
\dfrac{1}{2}\alpha^2 g^2_2(u_n)- \big(G_1(u_n)-G_1(u_1)\big)=E
+ \frac{h^2}{8}\big(g^2_1(u_n)-g^2_1(u_1)\big),
$$
holds, where $G_1$ is the primitive function of $g_1$ such that $G_1(0)=0$ and $E$ is defined as follow:
\begin{equation}\label{ec: error de trapecios}
E:=\sum_{k=1}^{n-1}\dfrac{g_1(u_k) + g_1(u_{k+1})}{2}(u_{k+1}-u_k)
- \big(G_1(u_n) - G_1(u_1)\big).
\end{equation}
It is easy to check that $E$ is the error of the approximation by the trapezoidal rule of the integral of the function $g_1$ in the interval $[u_1,u_n]$, considering the subdivision of $[u_1,u_n]$ defined by the nodes $u_1 ,\dots, u_n$. Moreover, taking into account that $g_1$ is a convex function in $\R_{\ge 0}$, we easily conclude that $E \ge 0$ holds. Therefore, from the previous considerations we deduce the following proposition, which is the discrete version of (\ref{ec: continuous energy conservation}).
\begin{proposition}\label{Prop: discrete energy conservation}
Let $(\alpha,u) \in(\R_{>0})^{n+1}$ be a solution of {\rm
(\ref{ec: sistema_A})}. Then
\begin{equation}\label{ec: discrete energy conservation}
\dfrac{1}{2}\alpha^2 g^2_2(u_n)- \big(G_1(u_n)-G_1(u_1)\big)=E
+ \frac{h^2}{8}\big(g^2_1(u_n)-g^2_1(u_1)\big),
\end{equation}
where $G_1$ is the primitive function of $g_1$ such that $G_1(0)=0$, and $E$ is defined as in {\rm (\ref{ec: error de trapecios})}. Furthermore, if we consider $E$ as a function of $u_1$ according to {\rm (\ref{ec: error de trapecios})}, where $u_k:=U_k(u_1)$ is defined as in {\rm (\ref{ec: relaciones recursivas})} for $2\le k\le n$, then $E$ is a positive increasing function over $\R_\ge 0$.
\end{proposition}
\begin{proof} For the above considerations, we only have to prove that $E$ is an increasing function over $\R_\ge 0$.

We consider $E$ as a function of $u_1$, where $u_k:=U_k(u_1)$ is defined as in {\rm (\ref{ec: relaciones recursivas})} for $2\le k\le n$. If rewrite $E$ as follows
\begin{equation}\label{ec: error de trapecios 2}
E = \sum_{k=1}^{n-1}\Big(\dfrac{g_1(u_k) + g_1(u_{k+1})}{2}(u_{k+1}-u_k)
- \big(G_1(u_{k+1}) - G_1(u_k)\big)\Big),
\end{equation}
then it suffices to show that each term of the previous sum is an increasing function over $\R_\ge 0$. In fact, fix $1 \le k \le n-1$; the derivative of the $k$th term of (\ref{ec: error de trapecios 2}) as function of $u_1$ is
$$\begin{array}{l}
 \dfrac{\partial}{\partial u_1} \Big(\dfrac{g_1(u_k) + g_1(u_{k+1})}{2}(u_{k+1}-u_k)
- \big(G_1(u_{k+1}) - G_1(u_k)\big)\Big)=\\[1.5ex]
\quad =\dfrac{g'_1(u_k) v'_k + g'_1(u_{k+1})v'_{k+1}}{2}(u_{k+1}-u_k)
+\dfrac{g_1(u_k)+ g_1(u_{k+1})}{2}(v'_{k+1}-v'_k)\\
\qquad -\big(g_1(u_{k+1})v'_{k+1} - g_1(u_k)v'_k\big),
\end{array}$$
where $v'_{k}:= U_k'(u_1) $ and $v'_{k+1}:=U_{k+1}'(u_1)$. Adding and subtracting $v'_k$ in each occurrence of $v'_{k+1}$ we obtain
$$\begin{array}{l}
 \dfrac{\partial}{\partial u_1} \Big(\dfrac{g_1(u_k) + g_1(u_{k+1})}{2}(u_{k+1}-u_k)
- \big(G_1(u_{k+1}) - G_1(u_k)\big)\Big)=\\[1.5ex]
\quad = \Big(\dfrac{g'_1(u_k) + g'_1(u_{k+1})}{2}(u_{k+1}-u_k) -\big(g_1(u_{k+1}) - g_1(u_k)\big)\Big) v'_k \\
\qquad + \Big(\dfrac{g'_1(u_{k+1})}{2}(u_{k+1}-u_k)
 - \dfrac{g_1(u_{k+1})- g_1(u_{k})}{2}\Big)(v'_{k+1}-v'_k)\\[1.5ex]
\quad = \Big(\dfrac{g'_1(u_k) + g'_1(u_{k+1})}{2}(u_{k+1}-u_k) -\big(g_1(u_{k+1}) - g_1(u_k)\big)\Big) v'_k \\
\qquad + \Big(\dfrac{g'_1(u_{k+1})}{2}
 - \dfrac{g'_1(\xi_k)}{2}\Big)(u_{k+1}-u_k)(v'_{k+1}-v'_k),
\end{array}$$
where $\xi_k \in (u_k,u_{k+1})$ is obtained after applying the Mean Value Theorem to $g_1(u_{k+1}) - g_1(u_k)$. It is easy to check that
$$
\dfrac{g'_1(u_k) + g'_1(u_{k+1})}{2}(u_{k+1}-u_k) -\big(g_1(u_{k+1}) - g_1(u_k)\big)
$$
is the error of the approximation by the trapezoidal rule of the integral of the function $g_1'$ in the interval $[u_{k}, u_{k+1}]$, and the convexity of $g_1'$ ensures their positivity. In the other hand, since $g_1'$ is increasing, we have that $g_1'(u_{k+1}) - g_1'(\xi_k) \ge 0$. Finally, from Lema \ref{Lema: positividad}, $(v_{k+1}'-v_k')$, $(u_{k+1}-u_k)$ and $v_k'$ are positive numbers for $u_1 \in \R_\ge 0$. Therefore, the $k$th term of (\ref{ec: error de trapecios 2}) is increasing over $\R_\ge 0$ for $1\le k \le n-1$, which completes the proof.
\end{proof}
%
%
\subsection{Bounds for the positive solutions}\label{subseccion: cotas}
In this section we show bounds for the positive solutions of (\ref{ec: sistema_A}). More precisely, we find an interval containing the positive solutions of (\ref{ec: sistema_A}) whose endpoints only depend on $\alpha$. These
bounds will allow us to establish an efficient procedure of approximation of this solution.

Let $g:\Rpos \rightarrow \Rpos$ and $G:\Rpos \rightarrow \Rpos$ be the functions defined by
\begin{equation}\label{ec: definicion g}
 g(x) := \frac{g_1}{g_2}(x),
\end{equation}
and
\begin{equation}\label{ec: definicion G}
 G(x) := \frac{G_1}{g_2^2}(x),
\end{equation}
where $G_1$ is the primitive function of $g_1$ such that $G_1(0)=0$.

As in \cite[Lemma 7]{Dratman11a} we have the following result
\begin{lemma}\label{Lema: relacion g1 g2 alpha}
Let $(\alpha,u) \in (\R_{>0})^{n+1}$ be a solution of {\rm
(\ref{ec: sistema_A})} for $A=\alpha$. Then
$$
\alpha g_2(u_n) < g_1(u_n).
$$
\end{lemma}
From Lemma \ref{Lema: relacion g1 g2 alpha} we obtain the following corollary.
\begin{corollary}\label{Coro: cota inf un alpha}
Let $(\alpha,u) \in (\R_{>0})^{n+1}$ be a solution of {\rm
(\ref{ec: sistema_A})} for $A=\alpha$. If the function $g$ defined in (\ref{ec: definicion g}) is surjective and strictly increasing, then
$$
u_n > g^{-1}(\alpha).
$$
\end{corollary}

Let $(\alpha,u) \in (\R_{>0})^{n+1}$ be a solution of {\rm (\ref{ec: sistema_A})} for $A=\alpha$. As in \cite[Lemma 9]{Dratman11a} we obtain an upper bound of $u_n$ in terms of $u_1$ and $\alpha$.

\begin{lemma}\label{Lema: cota sup un alpha u1}
Let $(\alpha,u) \in (\R_{>0})^{n+1}$ be a solution of {\rm (\ref{ec: sistema_A})} for $A=\alpha$, and let $C(\alpha)$ be an upper bound of $u_n$. Then $u_n < e^{M}u_1$ holds, with $M:=g_1'\big(C(\alpha)\big)$.
\end{lemma}

The next lemma shows a lower bound of $u_1$ in terms of $\alpha$.

\begin{lemma}\label{Lema: Cota inf u1 alpha}
Let $(\alpha,u) \in (\R_{>0})^{n+1}$ be a solution of {\rm (\ref{ec: sistema_A})} for $A=\alpha$, and let $C(\alpha)$ be an upper bound of $u_n$. If the function $g$ defined in (\ref{ec: definicion g}) is surjective and strictly increasing, then
$$
u_1 > \frac{g^{-1}(\alpha)}{e^M}
$$
holds, where $M:= g_1'\big(C(\alpha)\big)$.
\end{lemma}
\begin{proof}
From Lemma \ref{Lema: cota sup un alpha u1} and Corollary \ref{Coro: cota inf un alpha} we deduce
$$ g^{-1}(\alpha) < u_n < e^{M} u_1,$$
which immediately implies the statement of the lema.
\end{proof}

In the next lemma we obtain another upper bound of $u_n$ in terms of $u_1$ and $\alpha$. This upper bound will allow us to find an upper bound of $u_n$ in terms of $\alpha$.

\begin{lemma}\label{Lema: cota sup un alpha u1 bis}
Let $(\alpha,u) \in (\R_{>0})^{n+1}$ be a solution of {\rm (\ref{ec: sistema_A})} for $A=\alpha$. Then
$$
G(u_n) < G(u_1) + \frac{\alpha^2}{2},
$$
where $G$ is defined in (\ref{ec: definicion G}).

\noindent Moreover, if $G$ is surjective and strictly increasing, then
$$
u_n < G^{-1}\Big(G(u_1) + \frac{\alpha^2}{2}\Big).
$$

\end{lemma}
\begin{proof}
From Proposición \ref{Prop: discrete energy conservation} and Lema \ref{Lema: positividad}, we deduce the following inequality
$$G_1(u_n) - \frac{\alpha^2}{2}g_2^2(u_n) < G_1(u_1),$$
where $G_1$ is the primitive function of $g_1$ such that $G_1(0)=0$. Dividing for $g_2^2(u_n)$, we obtain
$$G(u_n) - \frac{\alpha^2}{2} < \frac{G_1(u_1)}{g_2^2(u_n)}.$$
Since $g_2^2$ is an increasing function, we conclude that
$$G(u_n)-\frac{\alpha^2}{2} < \frac{G_1(u_1)}{g^2_2(u_n)} \le G(u_1),$$
which prove the first part of the lemma.

Now, suppose that $G$ is surjective and strictly increasing, then $G$ is an invertible function and their inverse is strictly increasing. Combining this remark with the last inequality, we obtain
$$
u_n < G^{-1}\Big(G(u_1) + \frac{\alpha^2}{2}\Big),
$$
and the proof is complete.
\end{proof}

\noindent From this lemma we obtain upper bounds for $u_1$ and $u_n$ in terms of $\alpha$.
\begin{proposition}\label{Prop: cota sup u1 y un alpha}
Let $(\alpha,u) \in (\R_{>0})^{n+1}$ be a solution of {\rm (\ref{ec: sistema_A})} and let $g$ and $G$ be the functions defined in (\ref{ec: definicion g}) and (\ref{ec: definicion G}) respectively. Suppose that
\begin{itemize}
\item exists $d \in [0,1)$ such that $\big(\ln (G_1^d(x)/g^2_2(x))\big)' \ge 0$ for all $x>0$,
\item $G''(x) \ge 0$ for all $x>0$,
\end{itemize}
hold, where $G_1$ is the primitive function of $g_1$ such that $G_1(0)=0$. Then
$$
g^2(u_1) < \frac{\alpha^2}{1-d}.
$$
Moreover, if $g$ and $G$ are surjective functions, then
$$
u_1 < g^{-1}\Big(\frac{\alpha}{\sqrt{1-d}}\Big),
$$
and
$$
u_n < G^{-1}\Big(G\Big(g^{-1}\Big(\frac{\alpha}{\sqrt{1-d}}\Big)\Big) + \frac{\alpha^2}{2}\Big).
$$
\end{proposition}
\begin{proof} Combining Lemma \ref{Lema: cota sup un alpha u1} and the Mean Value Theorem, we deduce that exists $\xi > 0$ between $u_1$ and $u_n$ such that
$$
G'(\xi)(u_n - u_1)= G(u_n)-G(u_1)<\frac{\alpha^2}{2}.
$$
By Lemma \ref{Lema: positividad}(\ref{Lema: positividad - item 2}), we have that
$$
(u_n - u_1)= h^2\Big(\frac{n-1}{2}g_1(u_1) + \sum_{j=2}^{n-1}(n-j)g_1\big(u_j\big)\Big) > \frac{g_1(u_1)}{2}>0.
$$
Combining both inequalities, we obtain that
$$
G'(\xi) \frac{g_1(u_1)}{2}< G'(\xi)(u_n - u_1) < \frac{\alpha^2}{2}.
$$
Since $G''(x) \ge 0$ for all $x>0$, we see that $G'$ is an increasing function. Furthermore, we have the following inequality:
\begin{equation}\label{ec: 1er cota inf alpha^2}
\begin{array}{rcl}
\alpha^2 > G'(u_1)g_1(u_1) & = & \Big(\dfrac{g_1(u_1)g^2_2(u_1)-G_1(u_1)2g_2(u_1)g'_2(u_1)}{g^4_2(u_1)}\Big)g_1(u_1)\\[2ex]
  &=&\Big(1-\dfrac{G_1(u_1)2g_2(u_1)g'_2(u_1)}{g_1(u_1)g^2_2(u_1)}\Big)\dfrac{g^2_1(u_1)}{g^2_2(u_1)}\\[2ex]
  &=&\Big(1-\dfrac{\big(\ln(g^2_2(u_1))\big)'}{\big(\ln(G_1(u_1))\big)'}\Big)\dfrac{g^2_1(u_1)}{g^2_2(u_1)}.
\end{array}
\end{equation}
Taking into account the first condition of the statement, we deduce that
$$
\begin{array}{rcl}
0 \le \big(\ln(G^d_1(x)/g^2_2(x))\big)' &=& d\big(\ln(G_1(x))\big)'-\big(\ln(g^2_2(x))\big)'\\
&=&\big(\ln(G_1(x))\big)'\Big(d-\dfrac{\big(\ln(g^2_2(x))\big)'}{\big(\ln(G_1(x))\big)'}\Big).
\end{array}
$$
Since $\big(\ln(G_1(x))\big)'=g_1(x)/G_1(x) > 0$, we conclude that
\begin{equation}\label{ec: derivada log menor a c}
\dfrac{\big(\ln(g^2_2(x))\big)'}{\big(\ln(G_1(x))\big)'} \le d.
\end{equation}
Combining (\ref{ec: 1er cota inf alpha^2}) and (\ref{ec: derivada log menor a c}), we obtain
\begin{equation}\label{ec: 2da cota inf alpha^2}
\begin{array}{rcl}
\alpha^2 > (1-d)\dfrac{g^2_1(u_1)}{g^2_2(u_1)}=(1-d)g^2(u_1),
\end{array}
\end{equation}
and the first assertion of the proposition is proved.

Now, suppose that $g$ and $G$ are surjective functions. From (\ref{Lema: Crecimientos 2}), we have that $g$ and $G$ are strictly increasing functions. Combining this remark with (\ref{ec: 2da cota inf alpha^2}) and Lemma \ref{Lema: cota sup un alpha u1 bis}, we obtain the desired upper bounds for $u_1$ and $u_n$.
\end{proof}

Combining Proposition \ref{Prop: cota sup u1 y un alpha} and Lemma \ref{Lema: Cota inf u1 alpha} we obtain the following result.

\begin{lemma}\label{Lema: cota sup un alpha bis}
Let $(\alpha,u) \in (\R_{>0})^{n+1}$ be a solution of {\rm (\ref{ec: sistema_A})} and let $g$ and $G$ be the functions defined in (\ref{ec: definicion g}) and (\ref{ec: definicion G}) respectively. Suppose that
\begin{itemize}
\item $G$ and $g$  are surjective functions,
\item exists $d \in [0,1)$ such that $\big(\ln (G_1^d(x)/g_2^2(x))\big)' \ge 0$ for all $x>0$,
\item $G''(x) \ge 0$ for all $x>0$,
\end{itemize}
hold, where $G_1$ is the primitive function of $g_1$ such that $G_1(0)=0$. Then
$$
u_1 < g^{-1} \big(\alpha \widehat{C}(\alpha)\big),
$$
where
$$
\widehat{C}(\alpha):= 1 + \dfrac{g_2'\big(C_1(\alpha)\big) \alpha^2}{2g_2\big(g^{-1}(\alpha)/e^M\big)G'\big(g^{-1}(\alpha)/e^M\big)},
$$
with
$$
C_1(\alpha) := G^{-1}\Big(G\Big(g^{-1}\Big(\frac{\alpha}{\sqrt{1-d}}\Big)\Big) + \frac{\alpha^2}{2}\Big),
$$
and $M:=g'_1(C_1(\alpha))$.
Furthermore,
$$
u_n < G^{-1}\Big(G\Big(g^{-1}\Big(\alpha\widehat{C}(\alpha)\Big)\Big) + \frac{\alpha^2}{2}\Big).
$$
\end{lemma}
\begin{proof}
Let $(\alpha,u) \in (\R_{>0})^{n+1}$ be a solution of {\rm (\ref{ec: sistema_A})}. From Lemmas
\ref{Lema: positividad} and \ref{Lema: cota sup un alpha u1 bis} we deduce the inequalities
\begin{itemize}
\item $g_1(u_1)< h\Big(\dfrac{1}{2}g_1(u_1)+ g_1(u_2) + \cdots + g_1(u_{n-1}) + \dfrac{1}{2}g_1(u_n)\Big)=\alpha g_2(u_n),$
\item $ u_n <  G^{-1}\Big(G(u_1) + \frac{\alpha^2}{2}\Big).$
\end{itemize}
Combining both inequalities we obtain
$$
g(u_1)< \alpha \dfrac{g_2(u_n)}{g_2(u_1)} < \alpha \dfrac{g_2\Big(
G^{-1}\Big(G(u_1) + \frac{\alpha^2}{2}\Big)\Big)}{g_2(u_1)}.
$$
Since $g_2'$ is an increasing function in $\Rpos$ and $(G^{-1})'$ is a decreasing function in $\Rpos$, by the Mean Value Theorem, we obtain the following estimates
\begin{eqnarray*}
g(u_1)&<& \alpha \dfrac{g_2(u_1) + g_2'\Big(G^{-1}\Big(G(u_1) + \frac{\alpha^2}{2}\Big)\Big) \Big(G^{-1}\Big(G(u_1) + \frac{\alpha^2}{2}\Big)-u_1\Big)}{g_2(u_1)} \\
&<& \alpha \dfrac{g_2(u_1) + g_2'\Big(G^{-1}\Big(G(u_1) + \frac{\alpha^2}{2}\Big)\Big) (G^{-1})'(G(u_1))\frac{\alpha^2}{2}}{g_2(u_1)}\\
&<& \alpha \Big(1 + \dfrac{g_2'\Big(G^{-1}\Big(G(u_1) + \frac{\alpha^2}{2}\Big)\Big) \alpha^2}{2g_2(u_1)G'(u_1)}\Big).
\end{eqnarray*}
From Proposition \ref{Prop: cota sup u1 y un alpha} and Lemma \ref{Lema: Cota inf u1 alpha} we conclude that
$$
u_1 < g^{-1}\big(\alpha \widehat{C}(\alpha)\big),
$$
where
$$
\widehat{C}(\alpha):= 1 + \dfrac{g_2'\big(C_1(\alpha)\big) \alpha^2}{2g_2\big(g^{-1}(\alpha)/e^M\big)G'\big(g^{-1}(\alpha)/e^M\big)},
$$
with $M:=g'_1(C_1(\alpha))$ and
$$
C_1(\alpha) := G^{-1}\Big(G\Big(g^{-1}\Big(\frac{\alpha}{\sqrt{1-d}}\Big)\Big) + \frac{\alpha^2}{2}\Big).
$$
Combining this remark with Lemma \ref{Lema: cota sup un alpha u1 bis} we obtain
$$
u_n < G^{-1}\Big(G\Big(g^{-1}\Big(\alpha\widehat{C}(\alpha)\Big)\Big) + \frac{\alpha^2}{2}\Big),
$$
which immediately implies the statement of the lemma.
\end{proof}

%
%
\subsection{Existence and uniqueness}\label{subseccion: existencia_y_unicidad}
Let $P:(\Rpos)^2\to\R$ be the nonlinear map defined by
\begin{equation}\label{ec: funcion_G}
P(\alpha,u_1) := \mbox{$\frac{1}{h}$} \big(U_{n-1}(u_1)-U_n(u_1)\big) - \mbox{$\frac{h}{2}$} g_1\big(U_n(u_1)\big) + \alpha g_2\big(U_n(u_1)\big).
\end{equation}
Observe that $P(A,U_1)=0$ represents the minimal equation satisfied by the coordinates $(\alpha,u_1)$ of any (complex)
solution of the nonlinear system (\ref{ec: sistema_A}). Therefore, for fixed $\alpha \in \R_{>0}$, the positive roots of $P(\alpha,U_1)$ are the values of $u_1$ we want to obtain. Furthermore, from the parametrizations (\ref{ec: relaciones recursivas}) of the coordinates $u_2 \klk u_n$ of a given solution $(\alpha,u_1\klk u_n) \in (\Rpos)^{n+1}$ of (\ref{ec: sistema_A}) in terms of $u_1$, we conclude that the number of positive roots of $P(\alpha,U_1)$ determines the number of positive solutions of (\ref{ec: sistema_A}) for such a value of $\alpha$.

Since $P(A,U_1)$ is a continuous function in $(\R_{>0})^2$, as in \cite[Proposition 4]{Dratman11a} we have the following result:
\begin{proposition}\label{Prop: existencia}
Fix $\alpha> 0$ and $n \in \N$. If the function $g$ defined in (\ref{ec: definicion g}) is surjective, then {\rm (\ref{ec: sistema_A})} has a positive solution with $A=\alpha$.
\end{proposition}

In order to establish the uniqueness, we prove that the homotopy path that we obtain by moving the parameter $\alpha$ in $\Rpos$ is smooth. For this purpose, we show that the rational function $A(U_1)$ implicitly defined by the
equation $P(A,U_1)=0$ is increasing. We observe that an explicit expression for this function in terms of $U_1$ is obtained in (\ref{ec: relaciones recursivas}).
\begin{theorem}\label{Teo: derivada alpha positiva}
Let $\mathcal{A} > 0$ be a given constant and let $A(U_1)$ be the rational function of (\ref{ec: relaciones recursivas}). Let $g$ and $G$ be the functions defined in (\ref{ec: definicion g}) and (\ref{ec: definicion G}) respectively. Suppose that
\begin{itemize}
\item $G$ and $g$  are surjective functions,
\item exists $d \in [0,1)$ such that $\big(\ln (G_1^d(x)/g_2^2(x))\big)' \ge 0$ for all $x>0$,
\item $G''(x) \ge 0$ for all $x>0$,
\end{itemize}
hold, where $G_1$ is the primitive function of $g_1$ such that $G_1(0)=0$. Then there exists $M(\mathcal{A}) > 0$ such that the condition $A'(u_1) >0$ is satisfied for $n > 1 + M(\mathcal{A})/(2-2d)$ and $u_1 \in A^{-1}\big((0,\mathcal{A}]\big)\cap \R_{>0}$.
\end{theorem}
\begin{proof}
Let $U_1, U_2\klk U_n, A$ be the functions defined in (\ref{ec: relaciones recursivas}). For $u_1>0$, we denote by
$I(u_1) := G_1(U_n(u_1))-G_1(U_1(u_1))$ the integral of the function $g_1$ in $[U_1(u_1),U_n(u_1)]$, and by $T(u_1)$ the trapezoidal rule applied to $I(u_1)$, with the nodes $U_1(u_1),$ $ U_2(u_1),
\klk U_n(u_1)$. More precisely, $T$ is define as follows:
$$
T := \sum_{k=1}^{n-1}
\dfrac{g_1(U_{k+1})+g_1(U_k)}{2}(U_{k+1}-U_k).
$$
Finally, set $E:= T-I$. Combining Proposition \ref{Prop: discrete energy conservation} and the convexity of $g_1$, we deduce that $E>0$ and $E'>0$ in $\Rpos$, where $E'$ represent the derivative of $E$ with respect of $u_1$.

According to Proposition \ref{Prop: discrete energy conservation}, $U_1,U_2,\dots, U_n,A$ satisfy the discrete version (\ref{ec: discrete energy conservation}) of the energy conservation law (\ref{ec: continuous conservation law}). Dividing both sides of (\ref{ec: discrete energy conservation}) by $G_1(U_n)$ we obtain the following identities:
\begin{equation}\label{ec: discrete energy conservation bis}
\dfrac{1}{2}A^2 \dfrac{g^2_2(U_n)}{G_1(U_n)}= \dfrac{T}{G_1(U_n)} + \frac{h^2}{8} \dfrac{g^2_1(U_n)}{G_1(U_n)} \Big(1-\dfrac{g^2_1(U_1)}{g^2_1(U_n)}\Big).
\end{equation}
Taking derivatives with respect to $U_1$ at both sides of (\ref{ec: discrete energy conservation bis}), we have
\begin{equation}\label{ec: energia derivada}
\begin{array}{l}
A A' \dfrac{g^2_2(U_n)}{G_1(U_n)} + \dfrac{A^2}{2} \Big(\dfrac{g^2_2(U_n)}{G_1(U_n)}\Big)'= \\[2ex]
= \Big(\dfrac{T}{G_1(U_n)}\Big)' + \dfrac{h^2}{8} \Big(\dfrac{g^2_1(U_n)}{G_1(U_n)}\Big)'\Big(1-\dfrac{g^2_1(U_1)}{g^2_1(U_n)}\Big) - \dfrac{h^2}{8} \dfrac{g^2_1(U_n)}{G_1(U_n)}\Big(\dfrac{g^2_1(U_1)}{g^2_1(U_n)}\Big)'.
\end{array}
\end{equation}
Let $u_1 \in A^{-1}\big((0,\mathcal{A}]\big)\cap \Rpos$. By Lema \ref{Lema: positividad}, $U_i(u_1)$ and $U_i'(u_1)$ are positive for $1 \le i \le n$. Furthermore, $g_1$, $g_2$, $g$, $G$ and $G_1$ are positive and increasing functions in $\Rpos$. Throughout the demonstration we will use these conditions repeatedly.

From Lemmas \ref{Lema: Crecimientos 2} and \ref{Lema: Crecimientos} we deduce that ${g^2_1(U_n)}/{G_1(U_n)}$ is an increasing function and ${g^2_1(U_1)}/{g^2_1(U_n)}$ is a decreasing function. Combining these remarks with (\ref{ec: energia derivada}) we obtain
$$
\Big(A A' \dfrac{g^2_2(U_n)}{G_1(U_n)}\Big)(u_1) >
\bigg(\Big(\dfrac{T}{G_1(U_n)}\Big)' - \dfrac{A^2}{2}
\Big(\dfrac{g^2_2(U_n)}{G_1(U_n)}\Big)'\bigg)(u_1).
$$
we see that the inequality above may be rewritten in the form
\begin{eqnarray}
\Big(A A' \dfrac{g^2_2(U_n)}{G_1(U_n)}\Big)(u_1) &>&
\bigg(\dfrac{I^2}{G_1^{2}(U_n)}\Big(\dfrac{T}{I}\Big)' +
\dfrac{G_1^2(U_1)}{G_1^{2}(U_n)}\Big(\dfrac{T}{G_1(U_1)}\Big)'\bigg)(u_1) \nonumber\\
&& -\bigg( \dfrac{A^2}{2}
\Big(\dfrac{g^2_2(U_n)}{G_1(U_n)}\Big)'\bigg)(u_1). \label{ec: cota1 Aprima G-crec}
\end{eqnarray}
We claim that $\big({T}/{G_1(U_1)}\big)'(u_1)> 0$. Indeed, by
Lemmas \ref{Lema: positividad} and \ref{Lema: Crecimientos} we have that
$$
\Big(\dfrac{T}{g_1^2(U_1)}\Big)'(u_1) = \bigg(\sum_{k=1}^{n-1}
\dfrac{g_1(U_{k+1})+g_1(U_k)}{2g_1(U_1)}
h^2\Big(\dfrac{1}{2}+\sum_{j=2}^{k}
\dfrac{g_1(U_j)}{g_1(U_1)}\Big)\bigg)'(u_1) > 0.
$$
Combining this result with Lemma \ref{Lema: Crecimientos 2}, we conclude that
$$
\Big(\dfrac{T}{G_1(U_1)}\Big)'(u_1) = \Big(\dfrac{g_1^2(U_1)}{G_1(U_1)}
\dfrac{T}{g_1^2(U_1)}\Big)'(u_1) > 0.
$$
Combining the claim above with (\ref{ec: cota1 Aprima G-crec}) we deduce that
\begin{equation}\label{ec: cota2 Aprima G-crec}
\Big(A A' \dfrac{g^2_2(U_n)}{G_1(U_n)} \Big)(u_1) >
\bigg( \dfrac{I^2}{G_1^{2}(U_n)}\Big(\dfrac{T}{I}\Big)' - \dfrac{A^2}{2}
\Big(\dfrac{g^2_2(U_n)}{G_1(U_n)}\Big)'\bigg)(u_1).
\end{equation}
In order to prove the positivity of $A'(u_1)$, we rewrite the right side of (\ref{ec: cota2 Aprima G-crec}).
$$
\begin{array}{l}
\Big( \dfrac{I^2}{G_1^{2}(U_n)}\Big(\dfrac{T}{I}\Big)' - \dfrac{A^2}{2} \Big(\dfrac{g^2_2(U_n)}{G_1(U_n)}\Big)'\Big)(u_1) = \\[2ex]
= \bigg(\dfrac{T' I-T I'}{G_1^2(U_n)} + \dfrac{A^2}{2} \dfrac{g^2_2(U_n) \big(G_1(U_n)\big)'- \big(g_2^2(U_n)\big)'G_1(U_n)}{G^2_1(U_n)} \bigg)(u_1)\\[2ex]
= \Bigg( \dfrac{E' I-E I'}{G_1^2(U_n)} +
\dfrac{A^2g^2_2(U_n)\big(G_1(U_n)\big)'}{2G^2_1(U_n)} \bigg(1 -
\dfrac{\big(g_2^2(U_n)\big)'G_1(U_n)}{g^2_2(U_n)\big(G_1(U_n)\big)'}
\bigg)\Bigg)(u_1).
\end{array}
$$
Since there exists $d \in [0,1)$ such that $\big(\ln (G_1^d(x)/g_2^2(x))\big)' \ge 0$ for all $x>0$, we have that
$$
\Bigg(1 - \dfrac{\big(g_2^2(U_n)\big)'G_1(U_n)}{g^2_2(U_n)\big(G_1(U_n)\big)'}
\Bigg)(u_1) > 1-d.
$$
Furthermore, by the positivity of $E'(u_1)$ and the definition of $I(u_1)$, we conclude that $\big(E' I-E I'\big)(u_1)> -E(u_1) \big(G(U_n)\big)'(u_1)$. From these inequalities, we deduce that
\begin{eqnarray*}
\Big( \dfrac{I^2}{G_1^{2}(U_n)}\Big(\dfrac{T}{I}\Big)' - \dfrac{A^2}{2} \Big(\dfrac{g^2_2(U_n)}{G_1(U_n)}\Big)'\Big)(u_1) > \qquad\qquad\qquad\qquad\\
 \qquad\qquad\qquad\qquad > \Big( \dfrac{\big(G_1(U_n)\big)'}{G_1^2(U_n)}\Big( (1 - d) \dfrac{A^2g^2_2(U_n)}{2} - E \Big)\Big)(u_1).
\end{eqnarray*}
Combining these remarks with (\ref{ec: cota2 Aprima G-crec}), we obtain
$$
\Big(A A' \dfrac{g^2_2(U_n)}{G_1(U_n)} \Big)(u_1) >
\Bigg( \dfrac{\big(G_1(U_n)\big)'}{G_1^2(U_n)}\Big( (1 - d) \dfrac{A^2g^2_2(U_n)}{2} - E \Big)\Bigg)(u_1).
$$
From (\ref{ec: discrete energy conservation}), it follows that
{\small \begin{eqnarray}
\Big(A A' \dfrac{g^2_2(U_n)}{G_1(U_n)}\Big)(u_1) \!\!\!\!\!\!&>&\!\!\!\!\!\!
\Bigg(\dfrac{\big(G_1(U_n)\big)'}{G_1^2(U_n)} \dfrac{(1-d)A^2g^2_2(U_n)}{4}\Bigg)(u_1)\!\!\!\!\!\!
\label{ec: positividad Aprima Reduccion1}\\
&& \!\!\!\!\!\!+ \Bigg(\dfrac{\big(G_1(U_n)\big)'}{G_1^2(U_n)}\Big(\frac{1 \!-\! d}{2} T
+ \frac{1 \!-\! d}{2} \frac{h^2}{8}(g_1^2(U_n) \!-\! g_1^2(U_1)) \!-\! E
\Big)\Bigg)(u_1).\!\!\!\!\!\!\nonumber
\end{eqnarray}}
If we prove that the second term of the right side of (\ref{ec: positividad Aprima Reduccion1}) is
positive, we obtain that
\begin{equation}\label{ec: Cota Aprima G-crec}
\Big(A A' \dfrac{g^2_2(U_n)}{G_1(U_n)}\Big)(u_1) >
\Bigg(\dfrac{\big(G_1(U_n)\big)'}{G_1^2(U_n)}\dfrac{(1-d)A^2g^2_2(U_n)}{4}\Bigg)(u_1)
> 0,
\end{equation}
which immediately implies the statement of the theorem. Therefore, it suffices to show that
$$
\Bigg(\dfrac{\big(G_1(U_n)\big)'}{G_1^2(U_n)}\Big(\frac{1 - d}{2} T
+ \frac{1 - d}{2} \frac{h^2}{8}(g_1^2(U_n)-g_1^2(U_1)) - E
\Big)\Bigg)(u_1)>0.
$$
Since $g_1$ is an increasing function, we only need to show that
$$
\frac{1-d}{2} T(u_1) - E(u_1) = \sum_{k=1}^{n-1} \Big(\frac{1 - d}{2} T_k(u_1) - E_k(u_1)\Big)
> 0,
$$
where
\begin{eqnarray*}
T_k &:=& \dfrac{g_1(U_{k+1})+g_1(U_k)}{2}(U_{k+1}-U_k), \\
E_k &:=& T_k - I_k,
\end{eqnarray*}
with $I_k:=G_1(U_{k+1})-G_1(U_{k})$. Note that, for $u_1>0$, $I_k(u_1)$ is the integral of $g_1$ in $[U_{k}(u_1),U_{k+1}(u_1)]$, $T_k(u_1)$ is the trapezoidal rule applied to $I_k(u_1)$ and $E_k(u_1)$ is the error of such approximation.

In order to prove that $(1 - d) T(u_1)/2 - E(u_1) >0$, we show that
\begin{equation}\label{ec: positividad Aprima Reduccion2}
\frac{1 - d}{2} T_k(u_1) - E_k(u_1) >0
\end{equation}
holds for $1\le k \le n-1$. By \cite{DrAg98}, we have that
$$
E_k(u_1) \le \Big(\dfrac{g'_1(U_{k+1})+g'_1(U_k)}{8}(U_{k+1}-U_k)^2\Big)(u_1).
$$
From Lemma \ref{Lema: positividad} and the monotonicity of $g'_1$, we deduce that
\begin{eqnarray*}
E_k(u_1) &\le&
\Big(\dfrac{g'_1(U_{n})}{4}(U_{k+1}-U_k)^2 \Big)(u_1)\\
&\le& \bigg(h^2\dfrac{g'_1(U_{n})}{4} \sum_{j=1}^{k}\dfrac{g_1(U_{j+1})+g_1(U_j)}{2} (U_{k+1}-U_k)\bigg)(u_1)\\
&\le& \bigg(h \dfrac{g'_1(U_{n})}{4} \dfrac{g_1(U_{k+1})+g_1(U_k)}{2}
(U_{k+1}-U_k)\bigg)(u_1)\\
&& = \Big(h \dfrac{g'_1(U_{n})}{4} T_k\Big)(u_1).
\end{eqnarray*}
Thus, we see that (\ref{ec: positividad Aprima Reduccion2}) is satisfied if the inequality
\begin{equation}\label{ec: positividad Aprima Reduccion3}
h \dfrac{g'_1(U_{n})(u_1)}{4} \le \frac{1 - d}{2}
\end{equation}
holds. From Lemma \ref{Lema: cota sup un alpha bis} and the monotonicity of $g'_1$, we deduce that there exists a constant $M(\mathcal{A})> 0$ independent of $h$ such that $g'_1(U_{n})(u_1) \le M(\mathcal{A})$ for $u_1 \in A^{-1}((0,\mathcal{A}]) \cap \Rpos$. This shows that a sufficient condition for the fulfillment of (\ref{ec: positividad Aprima Reduccion3}), and thus of $A'(U_1) > 0$, is that $n-1 \ge
M(\mathcal{A})/(2-2d)$ holds. This finishes the proof of the theorem.
\end{proof}

In order to prove the uniqueness of positive solutions of (\ref{ec: sistema_A}), we still need a result
on the structure of the inverse image of $A$ on the interval under consideration.
\begin{lemma}\label{Lema: Preimagen=intervalo}
Let $\mathcal{A} > 0$ be a given constant and let $A(U_1)$ be the rational function of (\ref{ec: relaciones recursivas}). Let $g$ and $G$ be the functions defined in (\ref{ec: definicion g}) and (\ref{ec: definicion G}) respectively. Suppose that
\begin{itemize}
\item $G$ and $g$  are surjective functions,
\item exists $d \in [0,1)$ such that $\big(\ln (G_1^d(x)/g_2^2(x))\big)' \ge 0$ for all $x>0$,
\item $G''(x) \ge 0$ for all $x>0$,
\end{itemize}
hold, where $G_1$ is the primitive function of $g_1$ such that $G_1(0)=0$. Then there exists $M(\mathcal{A}) > 0$ that satisfies the following condition: for $n > 1 + M(\mathcal{A})/(2-2d)$ there exists $c:=c(n,\mathcal{A})>0$ such that
$A^{-1}((0,\mathcal{A}])\cap\R_{>0}=(0,c]$.
\end{lemma}
\begin{proof}
By Theorem \ref{Teo: derivada alpha positiva} we have that there exists $M(\mathcal{A}) > 0$ such that the condition $A'(u_1) >0$ is satisfied for $n > 1 + M(\mathcal{A})/(2-2d)$ and $u_1 \in A^{-1}\big((0,\mathcal{A}]\big)\cap \R_{>0}$. Fix $n \ge 1 + M(\mathcal{A})/(2-2d)$. From Lemma \ref{Lema: positividad} we deduce that $U_n(u_1)$ defines a bijective function in $\Rpos$ and that
$$
A(u_1) = \Big(\frac{h}{2}g_1(U_1(u_1))+\sum_{k=2}^{n-1}{h}g_1(U_k(u_1))+\frac{h}{2}g_1(U_n(u_1))\Big)/g_2(U_n(u_1)).
$$
Since $0< U_1(u_1) < \cdots <U_n(u_1)$, we have the following inequalities:
$$
\frac{h}{2} g\big(U_n(u_1)\big) \le A(u_1) \le g\big(U_n(u_1)\big).
$$
Since $\lim_{u_1 \rightarrow 0^+} g\big(U_n(u_1))=0$, there exists
$\epsilon > 0$ such that $(0,\epsilon] \subset A^{-1}((0,\mathcal{A}])
\cap \R_{>0}$. We claim that
$$
(0,c_0]=A^{-1}((0,\mathcal{A}])\cap\R_{>0}
$$
with
$$
c_0:=\sup\{\epsilon: (0, \epsilon] \subset A^{-1}((0,\mathcal{A}])
\cap \R_{>0}\}.
$$
Indeed, from the definition of $c_0$ we obtain that $(0,c_0) \subset
A^{-1}((0,\mathcal{A}]) \cap \R_{>0}$ and that $\lim_{u_1 \rightarrow
c_0^-} A(u_1) = A(c_0) \le \mathcal{A}$. Therefore, we deduce that
$$
(0,c_0] \subset A^{-1}((0,\mathcal{A}]) \cap \R_{>0}.
$$
We now show that the last set inclusion is an equality. Suppose that there exists $\delta > c_0$ such that $A(\delta) \le \mathcal{A}$. Let $c_1:=\inf\{\delta: \delta > c_0 , A(\delta) \le \mathcal{A} \}$. From the definition of $c_0$,
the interval $(c_0,c_1)$ is not empty. Since $A(x) > \mathcal{A}$ for all $x \in (c_0,c_1)$, we have that $A'(c_1) \le 0$, which contradicts the fact that $A'(u_1) > 0$ for all $u_1 \in A^{-1}((0,\mathcal{A}])\cap\R_{>0}$.
\end{proof}

Now we state and prove the main result of this section:
\begin{theorem}\label{Teo: unicidad}
Let $\alpha > 0$ be a given constant. Let $g$ and $G$ be the functions defined in (\ref{ec: definicion g}) and (\ref{ec: definicion G}) respectively. Suppose that
\begin{itemize}
\item $G$ and $g$  are surjective functions,
\item exists $d \in [0,1)$ such that $\big(\ln (G_1^d(x)/g_2^2(x))\big)' \ge 0$ for all $x>0$,
\item $G''(x) \ge 0$ for all $x>0$,
\end{itemize}
hold, where $G_1$ is the primitive function of $g_1$ such that $G_1(0)=0$. Then there exists $M(\alpha) > 0$ such that {\rm (\ref{ec: sistema_alpha})} has a unique positive solution for $n > 1 + M(\alpha)/(2-2d)$.
\end{theorem}
\begin{proof}
Proposition \ref{Prop: existencia} shows that (\ref{ec: sistema_alpha}) has solutions in $(\R_{>0})^n$ for any $\alpha
>0$ and any $n\in\N$. Therefore, there remains to show the uniqueness assertion.

By Theorem \ref{Teo: derivada alpha positiva} we have that there exists $M(\mathcal{A}) > 0$ such that the condition $A'(u_1) >0$ is satisfied for $n > 1 + M(\mathcal{A})/(2-2d)$ and $u_1 \in A^{-1}\big((0,\mathcal{A}]\big)\cap \R_{>0}$. From Lemma \ref{Lema: Preimagen=intervalo}, there exists $c=c(n,\alpha)$ such that $A^{-1}((0,\alpha])\cap \R_{>0}=(0,c]$. Arguing by contradiction, assume that there exist two distinct positive solutions $(u_1\klk u_n)$, $(\widehat{u}_1\klk\widehat{u}_n)\in(\Rpos)^{n}$ of (\ref{ec: sistema_alpha}). This implies that $u_1\not=\widehat{u}_1$ and $A(u_1)=A(\widehat{u}_1)$, where $A(U_1)$ is defined in (\ref{ec: relaciones recursivas}).
But this contradicts the fact that $A'(u_1)>0$ holds in $(0,c]$, showing thus the theorem.
\end{proof}
%
%
\section{Numerical conditioning}\label{seccion: condicionamiento}
Let be given $n\in\N$ and $\alpha^*>0$. Let $g$ and $G$ be the functions defined in (\ref{ec: definicion g}) and (\ref{ec: definicion G}) respectively. Suppose that
\begin{itemize}
\item $G$ and $g$  are surjective functions,
\item exists $d \in [0,1)$ such that $\big(\ln (G_1^d(x)/g_2^2(x))\big)' \ge 0$ for all $x>0$,
\item $G''(x) \ge 0$ for all $x>0$,
\end{itemize}
hold, where $G_1$ is the primitive function of $g_1$ such that $G_1(0)=0$. In order to compute the positive solution of (\ref{ec: sistema_A}) for this value of $n$ and $A=\alpha^*$, we shall consider (\ref{ec: sistema_A}) as a family of systems parametrized by the values $\alpha$ of $A$, following the positive real path determined by (\ref{ec: sistema_A}) when $A$ runs through a suitable interval whose endpoints are $\alpha_*$ and $\alpha^*$, where $\alpha_*$ be a positive constant independent of $h$ to be fixed in Section \ref{seccion: algoritmo}.

A critical measure for the complexity of this procedure is the condition number of the path considered, which is essentially determined by the inverse of the Jacobian matrix of (\ref{ec: sistema_A}) with respect to the variables $U_1 \klk U_n$, and the gradient vector of (\ref{ec: sistema_A}) with respect to the variable $A$ on the path. In this section we prove the invertibility of such Jacobian matrix, and obtain an explicit form of its inverse. Then we obtain an upper bound on the condition number of the path under consideration.

Let $F:=F(A,U):\R^{n+1}\to\R^n$ be the nonlinear map defined by the right-hand side of (\ref{ec: sistema_A}). In this section we
analyze the invertibility of the Jacobian matrix of $F$ with respect to the variables $U$, namely,
$$
J(A,U):=\frac{\partial F}{\partial{U}}(A,U) :=\left(
\begin{matrix}
  \Gamma_1 & -1       \\
   -1   & \ddots & \ddots\\
   &    \ddots & \ddots & -1\\
   & & -1 &   \Gamma_n
\end{matrix}\right),
$$
with $\Gamma_1:=1+\frac{1}{2} h^2 g_1'(U_1)$, $\Gamma_i:=2 + h^2g_1'(U_i)$ for $2 \le i \le n-1$ and $\Gamma_n:= 1 + \frac{1}{2}h^2g_1'(U_n)-hAg_2'(U_n)$.

We start relating the nonsingularity of the Jacobian matrix $J(\alpha,u)$ with that of the corresponding point in the
path determined by (\ref{ec: sistema_A}). Let $(\alpha,u) \in (\Rpos)^{n+1}$ be a solution of (\ref{ec: sistema_A}) for $A=\alpha$. Taking derivatives with respect to $U_1$ in (\ref{ec: relaciones recursivas}) and substituting $u_1$ for $U_1$ we obtain the following tridiagonal system:
$$\left(\begin{array}{ccccc}
\Gamma_1(u_1) & -1 &\\
-1 & \ddots& \ddots \\
&\ddots &\ddots & -1\\
&  & -1 & \Gamma_n(u_1)
\end{array}\right)
\left(\begin{array}{c}
1\\
U_2'(u_1)\\
\vdots\\
U_n'(u_1)
\end{array}\right)=
\left(\begin{array}{c}
0\\
\vdots\\
0\\
h g_2\big(U_n(u_1)\big) A'(u_1)
\end{array}\right).
$$
For $1\le k\le n-1$, we denote by $\Delta_k:=\Delta_k(A,U)$ the $k$th principal minor of the matrix $J(A,U)$, that is, the $(k \times k)$-matrix formed by the first $k$ rows and the first $k$ columns of $J(A,U)$. By the Cramer rule we deduce the identities:
\begin{eqnarray}
h g_2\big(U_n(u_1)\big) A'(u_1) &=& \det\big(J(\alpha,u)\big), \label{ec: A'=det 1}\\
\det\big(J(\alpha,u)\big) U_k'(u_1) &=& h g_2\big(U_n(u_1)\big) A'(u_1) \det\big(\Delta_{k-1}(\alpha,u)\big), \label{ec: A'=det 2}
\end{eqnarray}
for $2\le k\le n$.

Let $\alpha > 0$ be a given constant. Then Theorem \ref{Teo: derivada alpha positiva} asserts that $A'(u_1)>0$ holds. Combining this inequality with (\ref{ec: A'=det 1}) we conclude that $\det\big(J(\alpha,u)\big)>0$ holds. Furthermore, by (\ref{ec: A'=det 2}), we have
\begin{equation} \label{ec: U'=det}
U_k'(u_1) = \det \big(\Delta_{k-1}(\alpha , u)\big) \quad (2 \le k \le n).
\end{equation}
Combining Remark \ref{Lema: positividad}(\ref{Lema: positividad - item 4}) and (\ref{ec: U'=det}) it follows that
$\det\big(\Delta_k(\alpha, u)\big)>0$ holds for $1 \le k \le n-1$. As a consequence, we have that all the principal minors of the symmetric matrix $J(\alpha,u)$ are positive. Then the Sylvester criterion shows that $J(\alpha,u)$ is positive definite. These remarks allows us to prove the following result.
\begin{theorem}\label{Teo: J inversible}
Let $(\alpha,v)\in(\R_{>0})^{n+1}$ be a solution of (\ref{ec: sistema_A}) for $A=\alpha$. Let $g$ and $G$ be the functions defined in (\ref{ec: definicion g}) and (\ref{ec: definicion G}) respectively. Suppose that
\begin{itemize}
\item $G$ and $g$  are surjective functions,
\item exists $d \in [0,1)$ such that $\big(\ln (G_1^d(x)/g_2^2(x))\big)' \ge 0$ for all $x>0$,
\item $G''(x) \ge 0$ for all $x>0$,
\end{itemize}
hold, where $G_1$ is the primitive function of $g_1$ such that $G_1(0)=0$. Then there exists $M(\alpha) > 0$ such that the matrix $J(\alpha, u)$ is symmetric and positive definite for $n > 1 + M(\alpha)/(2-2d)$.
\end{theorem}

Having shown the invertibility of the matrix $J(\alpha,u)$ for every solution $(\alpha,u) \in (\Rpos)^{n+1}$ of (\ref{ec: sistema_A}), the next step is to obtain explicitly the corresponding inverse matrices $J^{-1}(\alpha,u)$. For this purpose, we establish a result on the structure of the matrix $J^{-1}(\alpha,u)$.
\begin{proposition}\label{Prop: factoriz J^{-1}}
Let $(\alpha,u) \in(\R_{>0})^{n+1}$ be a solution of (\ref{ec: sistema_A}). Let $g$ and $G$ be the functions defined in (\ref{ec: definicion g}) and (\ref{ec: definicion G}) respectively. Suppose that
\begin{itemize}
\item $G$ and $g$  are surjective functions,
\item exists $d \in [0,1)$ such that $\big(\ln (G_1^d(x)/g_2^2(x))\big)' \ge 0$ for all $x>0$,
\item $G''(x) \ge 0$ for all $x>0$,
\end{itemize}
hold, where $G_1$ is the primitive function of $g_1$ such that $G_1(0)=0$. Then there exists $M(\alpha) > 0$ such that the following matrix factorization holds:
$$
J^{-1}(\alpha,u) = \!\!
 \left(\begin{array}{ccccccc}
 1& \frac{1}{u_2'} & \frac{1}{u_3'} &\dots & \frac{1}{u_n'} \\[1ex]
  & 1 & \frac{u_2'}{u_3'}  & \dots &  \frac{u_2'}{u_n'} \\[1ex]
  &  & \ddots & \ddots & \vdots \\[1ex]
  &  & & 1 & \frac{u_{n-1}'}{u_n'} \\[1ex]
  &   &  &  & 1
 \end{array}
 \right)\!\!
 \left(\begin{array}{ccccccc}
  \frac{1}{u_2'} & \\[1ex]
  \frac{1}{u_3'} & \frac{u_2'}{u_3'} &  \\[1ex]
  \vdots &  \vdots & \ddots &  \\[1ex]
  \frac{1}{u_n'} & \frac{u_2'}{u_n'} & \dots & \frac{u_{n-1}'}{u_n'} &  \\[1ex]
  \frac{1}{d(J)} & \frac{u_2'}{d(J)}  & \dots & \frac{u_{n-1}'}{d(J)} & \frac{u_n'}{d(J)}
 \end{array}
 \right),
 $$
for $n > 1 + M(\alpha)/(2-2d)$, where $d(J) := \det\big(J(\alpha,u)\big)$ and $u_k':=U_k'(u_1)$ for $2 \le k \le n$.
\end{proposition}
\begin{proof}
Since $J(\alpha,u)$ is symmetric, invertible, tridiagonal and their $(n-1)$th principal minor is positive definite, the proof follows in the same way as that of \cite[Proposition 25]{DrMa09}.
\end{proof}
%
%

From the explicitation of the inverse of the Jacobian matrix $J(A,U)$ on the points of the real path
determined by (\ref{ec: sistema_A}), we can finally obtain estimates on the condition number of such a path.

Let $\alpha^*> 0$ and $\alpha_*>0$ constants independents of $h$ be given. Then Theorem \ref{Teo: unicidad} proves that (\ref{ec: sistema_A}) has a unique positive solution with $A=\alpha$ for every $\alpha$ in the real interval $\mathcal{I} := \mathcal{I}(\alpha_*, \alpha^*)$ whose endpoints are $\alpha_*$ and $\alpha^*$, which we denote by
$\big(u_1(\alpha), U_2\big(u_1(\alpha)\big) \klk U_n\big(u_1(\alpha)\big)\big)$. We bound the condition number
$$
\kappa := \max\{\|\varphi'(\alpha)\|_\infty : \alpha \in \mathcal{I}\},
$$
associated to the function $\varphi: \mathcal{I} \to \R^n$, $\varphi(\alpha):=\big(u_1(\alpha), U_2\big(u_1(\alpha)\big) \klk U_n\big(u_1(\alpha)\big)\big)$.

For this purpose, from the Implicit Function Theorem we have
\begin{eqnarray*}
\|\varphi'(\alpha)\|_\infty &=& \Big\|\Big(\frac{\partial F}{\partial U} \big(\alpha,\varphi(\alpha)\big)\Big)^{-1} \frac{\partial F}{\partial A}\big(\alpha,\varphi(\alpha)\big)\Big\|_\infty\\
&=& \Big\|J^{-1}\big(\alpha,\varphi(\alpha)\big) \frac{\partial F}{\partial A}\big(\alpha,\varphi(\alpha)\big)\Big\|_\infty.
\end{eqnarray*}
We observe that $(\partial F/\partial A)(\alpha, \varphi(\alpha))\!=\!\Big(0,\dots ,0, -h g_2\big(U_n\big(u_1(\alpha)\big)\big)\Big)^t$ holds. From Proposition \ref{Prop: factoriz J^{-1}} we obtain
$$
\|\varphi'(\alpha)\|_\infty = \Big\| \frac{h g_2\big(U_n\big(u_1(\alpha)\big)\big)} {\det\big(J\big(\alpha,\varphi(\alpha)\big)\big)} \Big(1,{U_2'\big(u_1(\alpha)\big)}, \dots ,  U_n'\big(u_1(\alpha)\big) \Big)^t\Big\|_\infty.
$$
Combining this identity with (\ref{ec: A'=det 1}), we conclude that
$$
\|\varphi'(\alpha)\|_\infty = \Big\| \frac{1} {A'\big(u_1(\alpha)\big)} \Big(1,{U_2'\big(u_1(\alpha)\big)}, \dots ,  U_n'\big(u_1(\alpha)\big) \Big)^t\Big\|_\infty.
$$
From Lemma \ref{Lema: positividad}, we deduce the following proposition.
\begin{proposition}\label{Prop: Nro condicion}
Let $\alpha^*> 0$ and $\alpha_*>0$ constants independents of $h$ be given. Let $g$ and $G$ be the functions defined in (\ref{ec: definicion g}) and (\ref{ec: definicion G}) respectively. Suppose that
\begin{itemize}
\item $G$ and $g$  are surjective functions,
\item exists $d \in [0,1)$ such that $\big(\ln (G_1^d(x)/g_2^2(x))\big)' \ge 0$ for all $x>0$,
\item $G''(x) \ge 0$ for all $x>0$,
\end{itemize}
hold, where $G_1$ is the primitive function of $g_1$ such that $G_1(0)=0$. Then there exists $M(\mathcal{I}) > 0$ such that
$$
\|\varphi'(\alpha)\|_\infty
=\frac{U_n'\big(u_1(\alpha)\big)}{A'\big(u_1(\alpha)\big)}
$$
holds for $\alpha \in \mathcal{I}$ and $n > 1 + M(\mathcal{I})/(2-2d)$.
\end{proposition}

\noindent Combining Proposition \ref{Prop: Nro condicion} and (\ref{ec: Cota Aprima G-crec}) we conclude that
$$
\begin{array}{rcl}
\|\varphi'(\alpha)\|_\infty &<& \dfrac{4 G_1\big(U_n\big(u_1(\alpha)\big)\big)}{(1-d) \alpha g_1\big(U_n\big(u_1(\alpha)\big)\big)}.
\end{array}
$$
Applying Lemma \ref{Lema: Cota inf u1 alpha} and Proposition \ref{Prop: cota sup u1 y un alpha} we deduce the
following result.
\begin{theorem}\label{Teo: Cota Nro Condicion}
Let $\alpha^*> 0$ and $\alpha_*>0$ constants independents of $h$ be given. Let $g$ and $G$ be the functions defined in (\ref{ec: definicion g}) and (\ref{ec: definicion G}) respectively. Suppose that
\begin{itemize}
\item $G$ and $g$  are surjective functions,
\item exists $d \in [0,1)$ such that $\big(\ln (G_1^d(x)/g_2^2(x))\big)' \ge 0$ for all $x>0$,
\item $G''(x) \ge 0$ for all $x>0$,
\end{itemize}
hold, where $G_1$ is the primitive function of $g_1$ such that $G_1(0)=0$. Then there exists a constant $\kappa_1(\alpha_*,\alpha^*) > 0$ independent of $h$ such that
$$
\kappa < \kappa_1(\alpha_*,\alpha^*).
$$
\end{theorem}
%
%
\section{An efficient numerical algorithm} \label{seccion: algoritmo}
As a consequence of the well conditioning of the positive solutions of (\ref{ec: sistema_A}), we shall exhibit an algorithm computing the positive solution of (\ref{ec: sistema_A}) for $A=\alpha^*$. This algorithm is a homotopy continuation method (see, e.g., \cite[\S 10.4]{OrRh70}, \cite[\S 14.3]{BlCuShSm98}) having a cost which is {\em linear} in $n$.

There are two different approaches to estimate the cost of our procedure: using Kantorovich--type estimates as in \cite[\S 10.4]{OrRh70}, and using Smale--type estimates as in \cite[\S 14.3]{BlCuShSm98}. We shall use the former, since we are able to control the condition number in suitable neighborhoods of the real paths determined by (\ref{ec: sistema_A}). Furthermore, the latter does not provide significantly better estimates.

Let $\alpha_* > 0$ be a constant independent of $h$. Let $g$ and $G$ be the functions defined in (\ref{ec: definicion g}) and (\ref{ec: definicion G}) respectively. Suppose that
\begin{itemize}
\item $G$ and $g$  are surjective functions,
\item exists $d \in [0,1)$ such that $\big(\ln (G_1^d(x)/g_2^2(x))\big)' \ge 0$ for all $x>0$,
\item $G''(x) \ge 0$ and $g''(x) \ge 0$ for all $x>0$,
\end{itemize}
hold, where $G_1$ is the primitive function of $g_1$ such that $G_1(0)=0$. Then the path defined by the positive solutions of (\ref{ec: sistema_A}) with $\alpha \in [\alpha_*, \alpha^*]$ is smooth, and the estimate of Theorem \ref{Teo: Cota Nro Condicion} hold. Assume that we are given a suitable approximation $u^{(0)}$ of the positive solution $\varphi(\alpha_*)$ of (\ref{ec: sistema_A}) for $A=\alpha_*$. In this section we exhibit an algorithm which, on input $u^{(0)}$, computes an approximation of $\varphi(\alpha^*)$. We recall that $\varphi$ denotes the function which maps each $\alpha > 0$ to the positive solution of (\ref{ec: sistema_A}) for $A=\alpha$. More precisely $\varphi : [\alpha_*,\alpha^*] \to \R^n$ is the function which maps each $\alpha \in [\alpha_*,\alpha^*]$ to the positive solution of (\ref{ec: sistema_A}) for $A=\alpha$, namely
$$
\varphi(\alpha):= \big(u_1(\alpha), \dots , u_n(\alpha)\big) := \big(u_1(\alpha), U_2\big(u_1(\alpha)\big),\dots, U_n\big(u_1(\alpha)\big)\big).
$$

From Lemma \ref{Lema: cota sup un alpha bis} and Lemma \ref{Lema: Cota inf u1 alpha}, we have that the coordinates of the positive solution of (\ref{ec: sistema_A}) tend to zero when $\alpha$ tends to zero. Therefore, for $\alpha$ small enough, we obtain a suitable approximation of the positive solution (\ref{ec: sistema_A}) for $A=\alpha_*$, and we track the positive real path determined by (\ref{ec: sistema_A}) until $A=\alpha^*$.

Let $0 < \alpha_* < \alpha^*$ be a constant independent of $h$ to be determined. Fix $\alpha \in [\alpha_*, \alpha^*]$. By Lemma \ref{Lema: cota sup un alpha bis} it follows that $\varphi(\alpha)$ is an interior point of the compact set
$$K_{\alpha}:=\{u \in \R^n : \|u\|_{\infty} \le 2 C_2(\alpha)\},$$
where
$$
C_2(\alpha) := G^{-1}\Big(G\Big(g^{-1}\Big({\alpha}{\widehat{C}(\alpha)}\Big)\Big) + \frac{\alpha^2}{2}\Big),
$$
with
$$
\widehat{C}(\alpha):= 1 + \dfrac{g_2'\big(C_1(\alpha)\big) \alpha^2}{2g_2\big(g^{-1}(\alpha)/e^M\big)G'\big(g^{-1}(\alpha)/e^M\big)},
$$
$$
C_1(\alpha) := G^{-1}\Big(G\Big(g^{-1}\Big(\frac{\alpha}{\sqrt{1-d}}\Big)\Big) + \frac{\alpha^2}{2}\Big),
$$
and $M:=g'_1(C_1(\alpha))$.

First we prove that the Jacobian matrix $J_\alpha(u):=({\partial F}/{\partial U})(\alpha,u)$ is invertible in a suitable subset of $K_{\alpha}$. Let $u \in \R^n$ and $v \in \R^n$ be points with
$$
\|u-\varphi(\alpha)\|_\infty < \delta_{\alpha}, \
\|v-\varphi(\alpha)\|_\infty < \delta_{\alpha},
$$
where $\delta_{\alpha} > 0$ is a constant to be determined. Note that if $\delta_\beta \le C_2(\alpha)$ then $u \in K_\alpha$ and $v \in K_\alpha$. By the Mean Value Theorem, we see that the entries of the diagonal matrix $J_\alpha(u)- J_\alpha(v)$ satisfy the estimates
$$
\begin{array}{lcll}
\Big|\big(J_\alpha(u)- J_\alpha(v)\big)_{ii} \Big| & \le & 2h^2
g_1''\big(2 C_2(\alpha)\big) \delta_{\alpha}, \ (1 \le i \le n-1)
\\[2ex]
\Big|\big(J_\alpha(u)- J_\alpha(v)\big)_{nn}\Big| & \le & 2{h}
\max\{{\alpha} g_2''\big(2C_2(\alpha)\big),g_1''\big(2
C_2(\alpha)\big)\} \delta_{\alpha}.
\end{array}
$$
By Theorem \ref{Teo: J inversible} and Proposition \ref{Prop: factoriz J^{-1}} we have that the matrix $J_{\varphi(\alpha)} := J_\alpha(\varphi(\alpha)) =({\partial F}/{\partial U})(\alpha ,\varphi(\alpha))$ is invertible and
$$\big(J_{\varphi(\alpha)}^{-1}\big)_{ij} =  \sum_{k=\max\{i,j\}}^{n-1} \frac{U_i'\big(u_1(\alpha)\big) U_j'\big(u_1(\alpha)\big)}{U'_k\big(u_1(\alpha)\big) U'_{k+1}\big(u_1(\alpha)\big)}+ \frac{U_i'\big(u_1(\alpha)\big) U_j'\big(u_1(\alpha)\big)}{U_n'\big(u_1(\alpha)\big)\det(J_{\varphi(\alpha)})} $$
holds for $1\le i, j \le n$. According to Lemma \ref{Lema: positividad}, we have $U_n'\big(u_1(\alpha)\big) \ge \cdots \ge U_2'\big(u_1(\alpha)\big) \ge 1$. These remarks show that
\begin{equation}\label{ec: Cota1 Exist Inv}
\begin{array}{l}
\Big\|J_{\varphi(\alpha)}^{-1}\big(J_\alpha(u)- J_\alpha(v)\big)\Big\|_{\infty} \le \\
\\
 \quad \le \eta_\alpha \delta_{\alpha} \bigg( 2  + \displaystyle\frac{h^2 +\sum_{j=2}^{n-1} h^2U'_j\big(u_1(\alpha)\big) + hU'_n\big(u_1(\alpha)\big)}{|\det(J_{\varphi(\alpha)})|}
 \bigg)\\
\\
 \quad \le 2\eta_\alpha \delta_{\alpha} \bigg( 1  + \displaystyle\frac{hU'_n\big(u_1(\alpha)\big)}{|\det(J_{\varphi(\alpha)})|}
 \bigg),
\end{array}
\end{equation}
where $\eta_\alpha:=2\max\{g_1''\big(2 C_2(\alpha)\big), \alpha g_2''\big(2 C_2(\alpha) \big) \}$. From (\ref{ec: A'=det 1}), we obtain the following identity:
$$\frac{hU'_n\big(u_1(\alpha)\big)}{|\det(J_{\varphi(\alpha)})|} = \frac{U'_n\big(u_1(\alpha)\big)\big)} {A'\big(u_1(\alpha)\big) g_2\big(u_n(\alpha)\big)}.$$
From (\ref{ec: Cota Aprima G-crec}), we have that
\begin{equation}\label{ec: Cota2 Exist Inv}
\frac{hU'_n\big(u_1(\alpha)\big)}{|\det(J_{\varphi(\alpha)})|} \!=\! \frac{U'_n\big(u_1(\alpha)\big)} {A'\big(u_1(\alpha)\big) g_2\big(u_n(\alpha)\big)} \!\le\!
\frac{4G_1\big(u_n(\alpha)\big)} {(1-d)g_1\big(u_n(\alpha)\big) g_2\big(u_n(\alpha)\big) A\big(u_1(\alpha)\big)}.
\end{equation}
From Lemma \ref{Lema: Crecimientos 2}, we have that $G'(x) > 0$ and $g'(x) > 0$ in $\Rpos$. Since $G$ is an increasing function, we deduce that
$$
\frac{G_1\big(u_n(\alpha)\big)} {g_1\big(u_n(\alpha)\big) g_2\big(u_n(\alpha)\big)}  <
\frac{1}{2 g_2'(u_n(\alpha))}.
$$
Combining the last inequality with (\ref{ec: Cota1 Exist Inv}) and (\ref{ec: Cota2 Exist Inv}), we obtain
$$
\Big\|J_{\varphi(\alpha)}^{-1}\big(J_\alpha(u)-
J_\alpha(v)\big)\Big\|_{\infty} \le
 2\eta_\alpha \delta_{\alpha} \bigg( 1  + \frac{2} {(1-d) g_2'\big(u_n(\alpha)\big) A\big(u_1(\alpha)\big)}
\bigg).
$$
From Corollary \ref{Coro: cota inf un alpha}, we have that
\begin{equation} \label{ec: Cota3 Exist Inv}
\Big\|J_{\varphi(\alpha)}^{-1}\big(J_\alpha(u)-
J_\alpha(v)\big)\Big\|_{\infty} \le  \bigg( \displaystyle\frac{4
\eta_\alpha (\theta^*+1)} {g_2'\big(g^{-1}(\alpha)\big) (1 - d)
\alpha}\bigg) \delta_{\alpha}.
\end{equation}
with $\theta^*:= g_2'\big(g^{-1}(\alpha^*)\big) (1 - d) \alpha^*/2$. Hence, defining $\delta_\beta$ in the following way:
\begin{equation}\label{ec: Def delta_alpha}
\delta_\alpha :=\min\Big\{ \displaystyle\frac{g_2'\big(g^{-1}(\alpha)\big) (1 - d) \alpha} {16 \eta_\alpha (\theta^*+1)}, C_2(\alpha) \Big\},
\end{equation}
we obtain
\begin{equation} \label{ec: Cota Exist Inv}
\Big\|J_{\varphi(\alpha)}^{-1}\Big(J_\alpha(u)- J_\alpha(v)\Big)\Big\|_{\infty} \le \frac{1}{4}.
\end{equation}
In particular, for $v=\varphi(\alpha)$, this bound allows us to consider $J_\alpha(u)$ as a perturbation of $J_{\varphi(\alpha)}$. More precisely, by a standard perturbation lemma (see, e.g., \cite[Lemma 2.3.2]{OrRh70}) we deduce that $J_\alpha(u)$ is invertible for every $u \in \mathcal{B}_{\delta_{\alpha}}(\varphi(\alpha))$ and we obtain the following upper bound:
\begin{equation} \label{ec: Cota Inv(u) por J}
\Big\|J_\alpha(u)^{-1} J_{\varphi(\alpha)}\Big\|_{\infty} \le \frac{4}{3}.
\end{equation}

In order to describe our method, we need a sufficient condition for the convergence of the standard Newton
iteration associated to (\ref{ec: sistema_A}) for any $\alpha \in [\alpha_*,\alpha^*]$. Arguing as in \cite[10.4.2]{OrRh70} we deduce the following remark, which in particular implies that the Newton iteration under consideration converges.
\begin{remark}\label{Remark: cond suf conv Newton}
Set $\delta:=\min\{ \delta_{\alpha}: \alpha \in [\alpha_*,\alpha^*] \}$. Fix $\alpha \in [\alpha_*,\alpha^*]$ and consider the Newton iteration
$$
u^{(k+1)} = u^{(k)}- J_\alpha(u^{(k)})^{-1} F(\alpha, u^{(k)}) \quad (k \ge 0), $$
starting at $u^{(0)} \in K_{\alpha}$. If $\|u^{(0)} - \varphi(\alpha)\|_\infty < \delta$, then
$$
\|u^{(k)}- \varphi(\alpha)\|_{\infty} < \frac{\delta}{3^{k}}
$$
holds for $k \ge 0$.
\end{remark}

Now we can describe our homotopy continuation method. Let
$\alpha_0:=\alpha_* < \alpha_1 < \cdots < \alpha_N:=\alpha^*$ be a uniform
partition of the interval $[\alpha_*,\alpha^*]$, with $N$ to be fixed. We
define an iteration as follows:
\begin{eqnarray}
u^{(k+1)} &=& u^{(k)} - J_{\alpha_k}(u^{(k)})^{-1} F(\alpha_k,u^{(k)}) \quad (0\le k\le N-1), \label{ec: 1ra iteracion}\\
u^{(N+k+1)} &=& u^{(N+k)} - J_{\alpha^*}(u^{(N+k)})^{-1} F(\alpha^*,u^{(N+k)}) \quad (k\ge 0). \label{ec: 2da iteracion}
\end{eqnarray}

In order to see that the iteration (\ref{ec: 1ra iteracion})--(\ref{ec: 2da iteracion}) yields an approximation of the positive solution $\varphi(\alpha^*)$ of (\ref{ec: sistema_A}) for $A=\alpha^*$, it is necessary to obtain a condition assuring that (\ref{ec: 1ra iteracion}) yields an attraction point for the Newton iteration (\ref{ec: 2da iteracion}). This relies on a suitable choice for $N$, which we now discuss.

By Theorem \ref{Teo: Cota Nro Condicion}, we have
\begin{eqnarray*}
 \|\varphi(\alpha_{i+1})-\varphi(\alpha_i)\|_{\infty} &\le& \max\{\|\varphi'(\alpha)\|_\infty : \alpha \in [\alpha_*,\alpha^*]\}\, |\alpha_{i+1}-\alpha_i| \\
 &\le& \kappa_1 \frac{\alpha^*}{N},
\end{eqnarray*}
for $0 \le i \le N-1$, where $\kappa_1$ is an upper bound of the condition number independent of $h$. Thus, for $N :=\lceil 3 \alpha^* \kappa_1 / \delta \rceil + 1 = O(1) $, by the previous estimate we obtain the following inequality:
\begin{equation}\label{ec: cota beta(i+1)-beta(i)}
\|\varphi(\alpha_{i+1})-\varphi(\alpha_i)\|_{\infty} < \frac{\delta}{3}
\end{equation}
for $0 \le i \le N-1$. Our next result shows that this implies the desired result.
\begin{lemma}\label{Lema: conv 1ra iteracion}
Set $N :=\lceil 3 \alpha^* \kappa_1 / \delta \rceil + 1$. Then, for every $u^{(0)}$ with $\|u^{(0)} - \varphi(\alpha_*)\|_\infty < \delta$, the point $u^{(N)}$ defined in (\ref{ec: 1ra iteracion}) is an attraction point for the Newton iteration (\ref{ec: 2da iteracion}).
\end{lemma}
\begin{proof}
By hypothesis, we have $\|u^{(0)}-\varphi(\alpha_*)\|_\infty < \delta$. Arguing inductively, suppose that
$\|u^{(k)}-\varphi(\alpha_k)\|_\infty < \delta$ holds for a given $0\le k < N$. By Remark \ref{Remark: cond suf conv Newton} we have that $u^{(k)}$ is an attraction point for the Newton iteration associated to (\ref{ec: sistema_A}) for $A=\alpha_k$. Furthermore, Remark \ref{Remark: cond suf conv Newton} also shows that $\|u^{(k+1)} - \varphi(\alpha_k)\|_\infty < \delta/3$ holds. Then
\begin{eqnarray*}
\|u^{(k+1)}-\varphi(\alpha_{k+1})\|_{\infty} &\le& \|u^{(k+1)}-\varphi(\alpha_{k})\|_{\infty} + \|\varphi(\alpha_{k})-\varphi(\alpha_{k+1})\|_{\infty} \\
&<& \mbox{$\frac {1}{3}$}\delta + \mbox{$\frac {1}{3}$}\delta  < \delta,
\end{eqnarray*}
where the inequality $\|\varphi(\alpha_{k+1})-\varphi(\alpha_{k})\|_{\infty} < \delta/3$ follows by (\ref{ec: cota beta(i+1)-beta(i)}). This completes the inductive argument and shows in particular that $u^{(N)}$ is an
attraction point for the Newton iteration (\ref{ec: 2da iteracion}).
\end{proof}

Next we consider the convergence of (\ref{ec: 2da iteracion}), starting with a point $u^{(N)}$ satisfying the condition $\|u^{(N)}-\varphi(\alpha^*)\|_\infty < \alpha \le \delta_{\alpha^*}$. Combining this inequality with (\ref{ec: Def delta_alpha}) we deduce that $u^{(N)} \in K_{\alpha^*}$. Furthermore, we see that
{\small \begin{equation}\label{ec: conv 2da iteracion}
\begin{array}{l}
\|u^{(N+1)} \!-\! \varphi(\alpha^*) \|_{\infty}\!=\! \|u^{(N)} \!-\! J_{\alpha^*}(u^{(N)})^{-1}
F(\alpha^*, u^{(N)}) \!-\! \varphi(\alpha^*)\|_{\infty}\\[1ex]
\qquad =\Big\|\! J_{\alpha^*}(u^{(N)})^{-1} \big(\!
J_{\alpha^*}(u^{(N)}) \big(u^{(N)} \!-\! \varphi(\alpha^*)\big) \!-\!
F(\alpha^*,u^{(N)}) \!+\! F(\alpha^*,\varphi(\alpha^*))\!\big)\Big\|_{\infty}\\[1ex]
\qquad \le \|\! J_{\alpha^*}(u^{(N)})^{-1}
J_{\varphi(\alpha^*)}\|_{\infty}\\
\qquad \qquad \qquad \Big\|J_{\varphi(\alpha^*)}^{-1} \big(\!
J_{\alpha^*}(u^{(N)}) \big(u^{(N)} \!-\! \varphi(\alpha^*)\big) \!-\!
F(\alpha^*,u^{(N)}) \!+\! F(\alpha^*,\varphi(\alpha^*))\!\big)\Big\|_{\infty}\\[1ex]
\qquad \le\! \|J_{\alpha^*}(u^{(N)})^{-1} J_{\varphi(\alpha^*)}\|_{\infty} \|J_{\varphi(\alpha^*)}^{-1}
\big(J_{\alpha^*}(u^{(N)})\!-\!J_{\alpha^*}(\xi)\big)\|_{\infty}
\|\big(u^{(N)} \!-\!\varphi(\alpha^*)\big)\|_{\infty},
\end{array}
\end{equation}}
where $\xi$ is a point in the segment joining the points $u^{(N)}$ and $\varphi(\alpha^*)$. Combining (\ref{ec: Cota3 Exist Inv}) and (\ref{ec: Cota Inv(u) por J}) we deduce that
$$\begin{array}{rcl}
\|u^{(N+1)}-\varphi(\alpha^*)\|_\infty &<&
\mbox{$\frac{4}{3}$} \big\|J_{\varphi(\alpha^*)}^{-1}
\big(J_{\alpha^*}(u^{(N)})\!-\!J_{\alpha^*}(\xi)\big)\big\|_{\infty} \delta_{\alpha^*}\\[1ex]
&<& \mbox{$\frac{4c}{3}$}\delta_{\alpha^*}^2 \le
\mbox{$\frac{1}{3}$} \delta_{\alpha^*},
\end{array}$$
holds, with $c\!:=\! \big(4 \eta_{\alpha^*} (\theta^*+1)\big)/\big(g_2'\big(g^{-1}(\alpha^*)\big) (1 - d) \alpha^* \big)$. By an inductive argument we conclude that the iteration (\ref{ec: 2da iteracion}) is well-defined and converges to the positive solution $\varphi(\alpha^*)$ of (\ref{ec: sistema_A}) for $A=\alpha^*$. Furthermore, we conclude that the point $u^{(N+k)}$, obtained from the point $u^{(N)}$ above after $k$ steps of the iteration (\ref{ec: 2da iteracion}), satisfies the estimate
$$ \|u^{(N+k)}-\varphi(\alpha^*)\|_{\infty} \le \hat{c} \big(\mbox{$\frac{4c}{3}$} \delta_{\alpha^*}\big)^{2^k} \le \hat{c} \Big(\frac{1}{3}\Big)^{2^k},$$
with $\hat{c}:={3}/{4 c }$. Therefore, in order to obtain an $\varepsilon$-approximation of $\varphi(\alpha^*)$, we have to perform $\log_2\log_3(3/4c\varepsilon)$ steps of the iteration (\ref{ec: 2da iteracion}). Summarizing, we have the following result.
\begin{lemma}\label{Lema: conv 2da iteracion}
Let $\varepsilon>0$ be given. Then, for every $u^{(N)} \in (\Rpos)^n$ satisfying the condition $\|u^{(N)}-
\varphi(\alpha^*)\|_\infty < \delta$, the iteration (\ref{ec: 2da iteracion}) is well-defined and the estimate
$\|u^{(N+k)}-\varphi(\alpha^*)\|_\infty < \varepsilon$ holds for $k \ge \log_2\log_3(3/4c\varepsilon)$.
\end{lemma}

Let $\varepsilon>0$. Assume that we are given $u^{(0)} \in (\Rpos)^n$ such that $\|u^{(0)}-\varphi(\alpha_*)\|_\infty < \delta$ holds. In order to compute an $\varepsilon$-approximation of the positive solution $\varphi(\alpha^*)$ of (\ref{ec: sistema_A}) for $A=\alpha^*$,  we perform $N$ iterations of (\ref{ec: 1ra iteracion}) and $k_0 := \lceil
\log_2\log_3(3/4c\varepsilon)\rceil$ iterations of (\ref{ec: 2da iteracion}). From Lemmas \ref{Lema: conv 1ra iteracion} and \ref{Lema: conv 2da iteracion} we conclude that the output $u^{(N+k_0)}$ of this procedure satisfies
the condition $\|u^{(N+k_0)} - \varphi(\alpha^*)\|_\infty < \varepsilon$. Observe that the Jacobian matrix $J_{\alpha}(u)$ is tridiagonal for every $\alpha \in [\alpha_*,\alpha^*]$ and every $u \in K_{\alpha}$. Therefore, the solution of a linear system with matrix $J_{\alpha}(u)$ can be obtained with $O(n)$ flops. This implies that each iteration of both (\ref{ec: 1ra iteracion}) and (\ref{ec: 2da iteracion}) requires $O(n)$ flops. In conclusion, we have the following result.
\begin{proposition}\label{Prop: Complejidad}
Let $\varepsilon > 0$ and $u^{(0)} \in (\Rpos)^n$ with $\|u^{(0)}-\varphi(\alpha_*)\|_\infty < \delta$ be given, where $\delta$ is defined as in Remark \ref{Remark: cond suf conv Newton}. Then the output of the iteration (\ref{ec: 1ra iteracion})--(\ref{ec: 2da iteracion}) is an $\varepsilon$-approximation of the positive solution $\varphi(\alpha^*)$ of (\ref{ec: sistema_A}) for $A=\alpha^*$. This iteration can be computed with $O(N n + k_0 n)=O\big(n \log_2\log_2(1/\varepsilon)\big)$ flops.
\end{proposition}
Finally, we exhibit a starting point $u^{(0)} \in (\Rpos)^n$ satisfying the condition of Proposition \ref{Prop: Complejidad}. Let $\alpha_* > 0$ be a constant independent of $h$ to be determined. We study the constant
$$\delta :=  \min\{ \delta_{\alpha}: \alpha \in [\alpha_*,\alpha^*] \},$$
where
$$
\delta_\alpha :=\min\Big\{ \displaystyle\frac{g_2'\big(g^{-1}(\alpha)\big) (1 - d) \alpha} {16 \eta_\alpha (\theta^*+1)}, C_2(\alpha) \Big\},
$$
with
$$
C_2(\alpha) := G^{-1}\Big(G\Big(g^{-1}\Big({\alpha}{\widehat{C}(\alpha)}\Big)\Big) + \frac{\alpha^2}{2}\Big),
$$
$$
\widehat{C}(\alpha):= 1 + \dfrac{g_2'\big(C_1(\alpha)\big) \alpha^2}{2g_2\big(g^{-1}(\alpha)/e^M\big)G'\big(g^{-1}(\alpha)/e^M\big)},
$$
$$
C_1(\alpha) := G^{-1}\Big(G\Big(g^{-1}\Big(\frac{\alpha}{\sqrt{1-d}}\Big)\Big) + \frac{\alpha^2}{2}\Big),
$$
and $M:=g'_1(C_1(\alpha))$.

Since $\widehat{C}(\alpha) \ge 1$, we have that
\begin{eqnarray*}
\eta_\alpha &=& 2\max\{g_1''\big(2 C_2(\alpha)\big), g_2''\big(2 C_2(\alpha)\big) \alpha \}\\
           & \le &  2\max\{g_1''\big(2 C_2(\alpha)\big), g_2''\big(2 C_2(\alpha)\big) g\big(G^{-1}\big(G\big(g^{-1}\big(\alpha\widehat{C}(\alpha)\big)\big)\big)\big) \},\\
           & \le &  2\max\{g_1''\big(2 C_2(\alpha)\big), g_2''\big(2 C_2(\alpha)\big) g\big(C_2(\alpha)\big) \},
\end{eqnarray*}
As $g_1$ and $g_2$ are analytic functions in $x=0$, in a neighborhood of $0 \in \R^n$, we obtain the following estimate:
$$
\eta_\alpha  \le  2 \max\{ S_1(\alpha), S_2(\alpha)\} \big(C_2(\alpha)\big)^{p-2},
$$
where $p$ is the multiplicity of 0 as a root of $g_1$ and $S_i$ is an analytic function in $x=0$ such that $\lim_{\alpha \rightarrow 0} S_i(\alpha) \neq 0$ for $i=1,2$. Taking into account that $\alpha \in (0, \alpha^*]$ holds, we conclude that there exists a constant $\eta^* > 0$, which depends only on $\alpha^*$, with
$$
\eta_\alpha \le 2 \eta^* \big(C_2(\alpha)\big)^{p-2}
$$
for all $\alpha \in (0, \alpha^*]$. Moreover, with a similar argument we deduce that there exists a constant $\vartheta^* > 0$, which depends only on $\alpha^*$, such that
\begin{equation}\label{ec: cota1 delta}
\delta_\alpha \ge \min \bigg\{ \displaystyle\frac{\vartheta^* (1 -
d)}{16 \eta^* (\theta^* + 1) } \Big(\displaystyle\frac{
g^{-1}(\alpha)}{ C_2(\alpha)}\Big)^{p-2}, \displaystyle\frac{
C_2(\alpha)} {g^{-1}(\alpha)} \bigg\} g^{-1}(\alpha).
\end{equation}
We claim that
\begin{equation}\label{ec: limite cociente 1}
\lim_{\alpha \rightarrow 0^+} \displaystyle\frac{ C_2(\alpha)}{ g^{-1}(\alpha)} = 1^+.
\end{equation}
In fact, since we have $\widehat{C}(\alpha) \ge 1$ and $g^{-1}$ is increasing, it follows that
\begin{equation}\label{ec: limite cociente 1 AUX1}
\displaystyle\frac{ C_2(\alpha)}{ g^{-1}(\alpha)} \ge 1.
\end{equation}
On the other hand, there exist $\xi_1 \in \Big(G\big(g^{-1}\big({\alpha}{\widehat{C}(\alpha)}\big)\big), G\big(g^{-1}\big({\alpha}{\widehat{C}(\alpha)}\big)\big) + \frac{\alpha^2}{2}\Big)$ and $\xi_2 \in \big(\alpha, \alpha \widehat{C}(\alpha)\big)$ with
\begin{eqnarray}
\displaystyle\frac{ C_2(\alpha)} { g^{-1}(\alpha)} & = & \displaystyle\frac{ g^{-1}\big({\alpha}{\widehat{C}(\alpha)}\big) + (G^{-1})'(\xi_1)\frac{\alpha^2}{2}}{ g^{-1}(\alpha)} \nonumber\\
& = & \displaystyle\frac{ g^{-1}({\alpha}) + (g^{-1})'(\xi_2) \alpha \big( \widehat{C}(\alpha) -1 \big)+ (G^{-1})'(\xi_1)\frac{\alpha^2}{2}}{ g^{-1}(\alpha)} \nonumber\\
& \le & 1+ \displaystyle\frac{ \alpha \big( \widehat{C}(\alpha) -1 \big)}{ g'\big(g^{-1}(\alpha)\big) g^{-1}(\alpha)}+ \displaystyle\frac{ \alpha^2}{2 G'\big(g^{-1}\big(\alpha\widehat{C}(\alpha)\big)\big) g^{-1}(\alpha)} \nonumber\\
& \le & 1+ \displaystyle\frac{ g\big(g^{-1}(\alpha)\big) \big( \widehat{C}(\alpha) -1 \big)}{ g'\big(g^{-1}(\alpha)\big) g^{-1}(\alpha)}+ \displaystyle\frac{ \big(g\big(g^{-1}(\alpha)\big)\big)^2}{2 G'\big(g^{-1}(\alpha)\big) g^{-1}(\alpha)}.\label{ec: limite cociente 1 AUX2}
\end{eqnarray}
Since $g_1$ and $g_2$ are analytic functions in $x=0$, we see that
$$
\lim_{\alpha \rightarrow 0^+} \displaystyle\frac{ g\big(g^{-1}(\alpha)\big) \big( \widehat{C}(\alpha) -1 \big)}{ g'\big(g^{-1}(\alpha)\big) g^{-1}(\alpha)} = 0,
$$
$$
\lim_{\alpha \rightarrow 0^+} \displaystyle\frac{ \big(g\big(g^{-1}(\alpha)\big)\big)^2}{2 G'\big(g^{-1}(\alpha)\big) g^{-1}(\alpha)}=0,
$$
Combining these remarks with (\ref{ec: limite cociente 1 AUX1}) and (\ref{ec: limite cociente 1 AUX2}) we inmediately deduce (\ref{ec: limite cociente 1}).

Combining (\ref{ec: cota1 delta}) with (\ref{ec: limite cociente 1}) we conclude that there exists a constant $C^* > 0$, which depends only on $\alpha^*$, with
$$
\delta_\alpha \ge C^* g^{-1}(\alpha).
$$
Therefore,
\begin{equation}\label{ec: cota2 delta}
\delta = \min\{\delta_\alpha : \alpha \in [\alpha_*, \alpha^*]\} \ge C^* g^{-1}(1/\alpha_*).
\end{equation}
From Lemma \ref{Lema: cota sup un alpha bis} and Lemma \ref{Lema: Cota inf u1 alpha}, we have
$$
    \varphi(\alpha_*) \in [g^{-1}(\alpha_*)/e^M, C_2(\alpha_*)]^n.
$$
Furthermore, by (\ref{ec: limite cociente 1}), we deduce that
\begin{equation}\label{ec: cota3 delta}
\bigg(\frac{C_2(\alpha_*)e^M} {g^{-1}(\alpha_*)} - 1 \bigg)
\frac{g^{-1}(\alpha_*)}{e^M} \le \bigg(\frac{C_2(\alpha_*)e^M}
{g^{-1}(\alpha_*)} - 1 \bigg) g^{-1}(\alpha_*) < C^*
g^{-1}(\alpha_*).
\end{equation}
holds for $\alpha_*>0$ small enough. Combining this with (\ref{ec: cota2 delta}), we conclude that
\begin{equation}\label{ec: (u - phi) menor delta}
\|u-\varphi(\alpha_*)\|_{\infty} \le C_2(\alpha_*) - g^{-1}(\alpha_*)/e^M < \delta
\end{equation}
holds for all $u \in [g^{-1}(\alpha_*)/e^M, C_2(\alpha_*)]^n$. Thus, let $\alpha_* <
\alpha^*$ satisfy (\ref{ec: cota3 delta}). Then, for any $u^{(0)}$ in the hypercube $[g^{-1}(\alpha_*)/e^M, C_2(\alpha_*)]^n$, the inequality
$$
\|u^{(0)}-\varphi(\alpha_*)\|_{\infty}  < \delta
$$
holds. Therefore, applying Proposition \ref{Prop: Complejidad}, we obtain our main result.
\begin{theorem}\label{Teo: final estimate cost}
Let $\varepsilon\!>\!0$ be given. Then we can compute an $\varepsilon$-approximation of the positive
solution of (\ref{ec: sistema_A}) for $A=\alpha^*$ with $O\big(n\log_2\log_2(1/\varepsilon)\big)$ flops.
\end{theorem}

\newcommand{\etalchar}[1]{$^{#1}$}
\providecommand{\bysame}{\leavevmode\hbox to3em{\hrulefill}\thinspace}
\providecommand{\MR}{\relax\ifhmode\unskip\space\fi MR }
\providecommand{\MRhref}[2]{%
  \href{http://www.ams.org/mathscinet-getitem?mr=#1}{#2}
}
\providecommand{\href}[2]{#2}

\end{document}